\documentclass[article]{IEEEtran}
\usepackage{cite}
\usepackage{color}
\usepackage[pdftex]{graphicx}
\usepackage{graphicx}
\graphicspath{{fig/}{jpeg/}}
\usepackage[cmex10]{amsmath}
\usepackage{amssymb}
\usepackage{algorithm}
\usepackage{algpseudocode}
\usepackage{amsmath,amssymb,lipsum}
\usepackage{hyperref}
\usepackage{comment}
\usepackage{graphicx}
\usepackage[font=small]{caption}
\usepackage{subcaption}
\usepackage{array}
\input{mysymbol.sty}
\usepackage{needspace}

\usepackage{amsthm}
\usepackage{tikz}
\usetikzlibrary{shapes,arrows}

\newtheorem{theorem}{Theorem}
\newtheorem{definition}{Definition}
\newtheorem{lemma}{Lemma}
\newtheorem{corollary}{Corollary}

\theoremstyle{definition}
\newtheorem{assumption}{Assumption}
\newtheorem{remark}{Remark}

\title{Navigation Functions for Convex Potentials in a Space with Convex Obstacles}
\author{Santiago Paternain, Daniel E. Koditschek and Alejandro Ribeiro
\thanks{Work supported by NSF CNS-1302222 and ONR N00014-12-1-0997. The authors are with the Department of Electrical and Systems Engineering, University of Pennsylvania. Email: \{spater, kod, aribeiro\}@seas.upenn.edu.}}

\begin{document}

\maketitle
\thispagestyle{empty}
\pagestyle{empty}


%
\begin{abstract}
Given a convex potential in a space with convex obstacles, an artificial potential is used to navigate to the minimum of the natural potential while avoiding collisions. The artificial potential combines the natural potential with potentials that repel the agent from the border of the obstacles. This is a popular approach to navigation problems because it can be implemented with spatially local information that is acquired during operation time. Artificial potentials can, however, have local minima that prevent navigation to the minimum of the natural potential. This paper derives conditions that guarantee artificial potentials have a single minimum that is arbitrarily close to the minimum of the natural potential. The qualitative implication is that artificial potentials succeed when either the condition number-- the ratio of the maximum over the minimum eigenvalue-- of the Hessian of the natural potential is not large and the obstacles are not too flat or when the destination is not close to the border of an obstacle. Numerical analyses explore the practical value of these theoretical conclusions.\end{abstract}
%
%
\section{Introduction}
It is customary in navigation problems to define the task of a robot as a given goal in its configuration space; e.g. \cite{choset2005principles, lavalle2006planning}. A drawback of this approach is the need for global information to provide the goal configuration. In a hill climbing problem, for instance, this means that the position of the top of the hill must be known, 
when it is more reasonable to assume that the robot senses its way to the top. 
In general, the ability to localize the source of a specific signal can be used by mobile robots to perform complex missions such as environmental monitoring
\cite{ogren2004cooperative, sukhatme2007design}, surveillance and reconnaissance \cite{rybski2000team}, and search and rescue operations \cite{kumar2004robot}. In all these scenarios the desired configuration is not available beforehand but a high level task is nonetheless well defined through the ability to sense the environment.

These task formulations can be abstracted by defining goals that minimize a convex potential, or equivalently, maximize a concave objective. The potential is unknown a priori but its values and, more importantly, its gradients can be estimated from sensory inputs. The gradient estimates derived from sensory data become inputs to a gradient controller that drives the robot to the potential's minimum if it operates in an open convex environment, e.g \cite{hirsch2004differential, tan2010extremum}. These gradient controllers are appealing not only because they exploit sensory information without needing an explicit target configuration, but also because of their simplicity and the fact that they operate using local information only.

This paper considers cases where the configuration space is not convex because it includes a number of nonintersecting convex obstacles. The goal is to design a modified gradient controller that relies on local observations of the objective function and local observations of the obstacles to drive the robot to the minimum of the potential while avoiding collisions. Both, objective function and obstacle observations are acquired at operation time. As a reference example  think of navigation towards the top of a wooded hill. The hill is modeled as a concave potential and the trunks a set of nonintersecting convex punctures. The robot is equipped with an inertial measurement unit (IMU) providing the slope's directional derivative, a GPS to measure the current height and a lidar unit  giving range and bearing to nearby physical obstacles \cite{johnson2011autonomous, Ilhan_Johnson_Koditschek_2015}. We then obtain local gradient measurement from the IMU, local height measurements from the GPS and local models of observed obstacles from the lidar unit and we want to design a controller that uses this spatially local information to drive the robot to the top of the hill. 

A possible solution to this problem is available in the form of artificial potentials, which have been widely used in navigation problems, see e.g. \cite{koditschek1990robot,  rimon1992exact, khatib1980commande, 
  filippidis2012navigation, filippidis2013navigation, filippidis2011navigation, filippidis2011adjustable, loizou2012navigation, loizou2011closed, lionis2007locally, lionis2008towards, roussos2013decentralized}. The idea is to mix the attractive potential to the goal configuration with repulsive artificial fields that push the robot away from the obstacles. This combination of potentials is bound to yield a function with multiple critical points. However, we can attempt to design combinations in which all but one of the critical points are saddles with the remaining critical point being close to the minimum of the natural potential. If this is possible, a gradient controller that follows this artificial potential reaches the desired target destination while avoiding collisions with the obstacles for almost all initial conditions (see Section \ref{sec_problem_formulation}).  

The design of mechanisms to combine potentials that end up having a unique minimum has been widely studied when the natural potential is rotationally symmetric. Koditschek-Rimon artificial potentials are a common alternative that has long been known to work for spherical quadratic potentials and spherical holes \cite{koditschek1990robot} and more recently generalized to focally admissible obstacles \cite{filippidis2013navigation}. In the case of spherical worlds local constructions of these artificial potentials have been provided in \cite{filippidis2011adjustable}. Further relaxations to these restrictions rely on the use of diffeomorphisms that map more generic environments. Notable examples are Koditschek-Rimon potentials in star shaped worlds \cite{rimon1992exact, rimon1991construction} and artificial potentials based on harmonic functions for navigation of topological complex three dimensional spaces \cite{loizou2011closed, loizou2012navigation}. These efforts have proven successful but can be used only when the space is globally known since this information is needed to design shuch diffeomorphism. Alternative solutions that are applicable without global knowledge of the environment are the use of polynomial navigation functions\cite{lionis2007locally} for n-dimensional configuration spaces with spherical obstacles and \cite{lionis2008towards} for 2-dimensional spaces with convex obstacles, as well as adaptations used for collision avoidance in multiagent systems  \cite{Tanner-RSS-05, dimarogonas2006totally, roussos2013decentralized}.

Perhaps the most comprehensive development in terms of expanding the applicability of artificial potentials is done in \cite{filippidis2012navigation, filippidis2013navigation, filippidis2011navigation}. This series of contributions reach the conclusion that Koditschek-Rimon potentials can be proven to have a unique minimum in spaces much more generic than those punctured by spherical holes. In particular it is possible to navigate any environment that is sufficiently curved. This is defined as situations in which the goals are sufficiently far apart from the borders of the obstacles as measured relative to their flatness. These ideas provides a substantive increase in the range of applicability of artificial potentials as they are shown to fail only when the obstacles are very flat or when the goal is very close to some obstacle border. 

Spherical quadratic potentials appear in some specific applications but are most often the result of knowing the goal configuration. Thus, the methods in \cite{koditschek1990robot,  rimon1992exact, khatib1980commande, 
  filippidis2012navigation, filippidis2013navigation, filippidis2011navigation, filippidis2011adjustable, loizou2012navigation, loizou2011closed, lionis2007locally, lionis2008towards, roussos2013decentralized} are applicable, for the most part, when the goal is known a priori and not when potential gradients are measured during deployment. To overcome this limitation, this work extends the theoretical convergence guarantees of Koditscheck-Rimon functions to problems in which the attractive potential is an arbitrary strongly convex function and the free space is a convex set with a finite number of nonintersecting smooth and strongly convex obstacles (Section \ref{sec_problem_formulation}) under mild conditions (Section \ref{sec_navigation_function}). The qualitative implication of these general conditions is that artificial potentials have a unique minimum when one of the following two conditions are met (Theorem \ref{theo_general}): (i) The condition number of the Hessian of the natural potential is not large and the obstacles are not too flat. (ii) The distance from the obstacles' borders to the  minimum of the natural potential is large relative to the size of the obstacles. These conditions are compatible with the definition of sufficiently curved worlds in \cite{filippidis2011navigation}. To gain further insight we consider the particular case of a space with ellipsoidal obstacles (Section \ref{sec_ellipsoidal_obstacles}). In this scenario the condition to avoid local minima is to have the minimum of the natural potential sufficiently separated from the border of all obstacles as measured by the product of the condition number of the objective and the eccentricity of the respective ellipsoidal obstacle (Theorem \ref{theo_ellipses}). The influence on the eccentricity of the obstacles had already been noticed in \cite{filippidis2012navigation,filippidis2011navigation}, however the results of Theorem \ref{theo_ellipses} refine those of the literature by providing an algebraic expression to check focal admissibility of the surface. 

Results described above are characteristics of the navigation function. The construction of a modified gradient controller that utilizes local observations of this function to navigate to the desired destination is addressed next (Section \ref{sec_navigation_function_unknown}). Convergence of a controller that relies on availability of local gradient observations of the natural potential and a local model of the obstacles is proven under the same hypothesis that guarantee the existence of a unique minimum of the potential function (Theorem \ref{theo_switched}). The local obstacle model required for this result assumes that only obstacles close to the agent are observed and incorporated into the navigation function but that once an obstacle is observed its exact form becomes known. In practice, this requires a space with sufficient regularity so that obstacles can be modeled as members of a class whose complete shape can be estimated from observations of a piece. In, e.g., the wooded hill navigation problem this can be accomplished by using the lidar measurements to fit a circle or an ellipse around each of the tree trunks. The practical implications of these theoretical conclusions are explored in numerical simulations (Section \ref{sec_numerical_examples}).


%
\section{Problem formulation}\label{sec_problem_formulation}
We are interested in navigating a punctured space while reaching a target point defined as the minimum of a convex potential function. Formally, let $\mathcal{\mathcal{X}} \in \mathbb{R}^n$ be a non empty compact convex set and let $f_0: \mathcal{\mathcal{X}} \rightarrow \mathbb{R}_+$ be a convex function whose minimum is the agent's goal. Further consider a set of obstacles $\mathcal{O}_i \subset \mathcal{\mathcal{X}}$ with $i=1\ldots m$ which are assumed to be open convex sets with nonempty interior and smooth boundary $\partial \ccalO_i$. The free space, representing the set of points accessible to the agent, is then given by the set difference between the space $\mathcal{X}$ and the union of the obstacles $\mathcal{O}_i$,
\begin{equation}\label{eqn_free_space_set}
   \mathcal{F} \triangleq \mathcal{X} \setminus \bigcup_{i=1}^m \mathcal{O}_i.
\end{equation}
The free space in \eqref{eqn_free_space_set} represents a convex set with convex holes; see, e.g., Figure 
\ref{fig_egg}. We assume here that the optimal point is in $\text{int}(\ccalF)$ of free space. Further let $t \in [0,\infty)$ denote a time index and let $x^*:= \argmin_{x \in \mathbb{R}^n} f_0(x) $. Then, the problem of interest is to generate a trajectory $x(t)$ that remains in the free space and reaches $x^*$ at least asymptotically,
\begin{equation}\label{eqn_navigation_goal}
   x(t) \in \mathcal{F},         \ 
   \forall t\in[0,\infty),       \quad 
   \mbox{and}                    \quad 
   \lim_{t\to \infty} x(t) = x^*.
\end{equation} 
%
%
In the canonical problem of navigating a convex objective defined over a convex set with a fully controllable agent, \eqref{eqn_navigation_goal} can be assured by defining a trajectory that varies along the negative gradient of the objective function, 
\begin{equation}\label{eqn_dynamical_gradient}
\dot{x} = -\nabla f_0(x).
\end{equation}
In a space with convex holes, however, the trajectories arising from the dynamical system defined by \eqref{eqn_dynamical_gradient} satisfy the second goal in \eqref{eqn_navigation_goal} but not the first because they are not guaranteed to avoid the obstacles. We aim here to build an alternative function $\varphi(x)$ such that the trajectory defined by the negative gradient of $\varphi(x)$ satisfies both conditions. It is possible to achieve this goal, if the function $\varphi(x)$ is a navigation function whose formal definition we introduce next \cite{koditschek1990robot}.

%
\begin{definition}[\bf{Navigation Function}]\label{def_navigation_function} 
Let $\mathcal{F} \subset \mathbb{R}^n$ be a compact connected analytic manifold with boundary. A map $\varphi : \mathcal{F} \rightarrow [0,1]$, is a navigation function in $\ccalF$ if:
\begin{mylist}
\item[{\bf Differentiable.}] It is twice continuously differentiable in $\mathcal{F}$.
\item[{\bf Polar at $x^*$.}] It has a unique minimum at $x^*$ which belongs to the interior of the free space, i.e., $x^* \in \mbox{int}(\mathcal{F})$.
\item [{\bf Morse.}] It has non degenerate critical points on $\ccalF$.
\item [{\bf Admissible.}] All boundary components have the same maximal value, namely $\partial \mathcal{F} = \varphi^{-1}(1)$.
\end{mylist}
\end{definition}
 
%
The properties of navigation functions in Definition \ref{def_navigation_function} are such that the solutions of the controller $\dot{x} = -\nabla \varphi(x)$ satisfy \eqref{eqn_navigation_goal} for almost all initial conditions. To see why this is true observe that the trajectories arising from gradient flows of a function $\varphi$, converge to the critical points and that the value of the function along the trajectory is monotonically decreasing,
\begin{equation}\label{eqn_decrasing}
\varphi(x(t_1)) \geq \varphi(x(t_2)), \quad \mbox{for any} \quad t_1<t_2.
\end{equation}
Admissibility, combined with the observation in \eqref{eqn_decrasing}, ensures that every trajectory whose initial condition is in the free space remains on free space for all future times, thus satisfying the first condition in \eqref{eqn_navigation_goal}. For the second condition observe that, as per \eqref{eqn_decrasing}, the only trajectory that can have as a limit set a maximum, is a trajectory starting at the maximum itself. This is a set of zero measure if the function satisfies the Morse property. Furthermore, if the function is Morse, the set of initial conditions that have a saddle point as a limit is the stable manifold of the saddle which can be shown to have zero measure as well. It follows that the set of initial conditions for which the trajectories of the system converge to the local minima of $\varphi$ has measure one. If the function is polar, this minimum is $x^*$ and the second condition in \eqref{eqn_navigation_goal} is thereby satisfied. We formally state this result in the next Theorem.

%
\begin{theorem}\label{theo_gradient_descent}
Let $\varphi$ be a navigation function on $\mathcal{F}$ as per Definition \ref{def_navigation_function}. Then, the flow given by the gradient control law
\begin{equation}\label{eqn_gradient_flow}
  \dot{x} = - \nabla \varphi(x),
\end{equation}
has the following properties:
\begin{description}
\item[(i)] $\mathcal{F}$ is a positive invariant set of the flow.
\item[(ii)] The positive limit set of $\mathcal{F}$ consists of the critical points of $\varphi$.
\item[(iii)] There is a set of measure one, $\tilde{\mathcal{F}} \subset \mathcal{F}$, whose limit set consists of $x^*$.
\end{description}
\end{theorem}

\begin{proof} See \cite{koditschek1988strict}. \end{proof}

%
Theorem \ref{theo_gradient_descent} implies that if $\varphi(x)$ is a navigation function as defined in \ref{def_navigation_function}, the trajectories defined by \eqref{eqn_gradient_flow} are such that $x(t)\in \mathcal{F}$ for all $t\in[0,\infty)$ and that the limit of $x(t)$ is the minimum $x^*$ for almost every initial condition. This means that \eqref{eqn_navigation_goal} is satisfied for almost all initial conditions. We can therefore recast the original problem \eqref{eqn_navigation_goal} as the problem of finding a navigation function $\varphi(x)$. Observe that Theorem \ref{theo_gradient_descent} guarantees that a navigation function can be used to drive a fully controllable agent [cf. \eqref{eqn_gradient_flow}]. However, navigation functions can also be used to drive agents with nontrivial dynamics as we explain in Remark \ref{rmk_dynamics}. 

%
To construct a navigation function $\varphi(x)$ it is convenient to provide a different characterization of free space. To that end, let $\beta_0: \mathbb{R}^n \to \mathbb{R}$ be a twice continuously differentiable concave function such that 
\begin{equation}\label{eqn_beta0}
\mathcal{X} = \left\{x \in \mathbb{R}^n \given \beta_0(x)\geq 0 \right\}.
\end{equation}
Since the function $\beta_0$ is assumed concave its super level sets are convex. Since the set $\mathcal{X}$ is also convex a function satisfying \eqref{eqn_beta0} can always be found. The boundary $\partial\mathcal{X}$, which is given by the set of points for which $\beta_0(x)=0$, is called the external boundary of free space. Further consider the $m$ obstacles $\ccalO_i$ and define $m$ twice continuously differentiable convex functions $\beta_i:\mathbb{R}^n \rightarrow \mathbb{R}$ for $i=1 \ldots m$. The function $\beta_i$ is associated with obstacle $\ccalO_i$ and satisfies
\begin{equation}\label{eqn_obstacle}
\mathcal{O}_i = \left\{ x \in \mathbb{R}^n \given \beta_i(x) < 0\right\}.
\end{equation}
Functions $\beta_i$ exist because the sets $\ccalO_i$ are convex and the sublevel sets of convex functions are convex. 

Given the definitions of the $\beta_i$ functions in \eqref{eqn_beta0} and \eqref{eqn_obstacle}, the free space $\ccalF$ can be written as the set of points at which all of these functions are nonnegative. For a more succinct characterization, define the function $\beta: \mathbb{R}^n \rightarrow \mathbb{R}$ as the product of the $m+1$ functions $\beta_i$,
\begin{equation}\label{eqn_obstacle_function}
\beta(x) \triangleq \prod_{i=0}^m \beta_i(x).
\end{equation}
If the obstacles do not intersect, the function $\beta(x)$ is nonnegative if and only if all of the functions $\beta_i(x)$ are nonnegative. This means that $x\in\ccalF$ is equivalent to $\beta(x)\geq0$ and that we can then define the free space as the set of points for which $\beta(x)$ is nonnegative -- when objects are nonintersecting. We state this assumption and definition formally in the following.

%
\begin{assumption}[\bf{Objects do not intersect}] \label{assum_obstacles}
Let $x\in\mathbb{R}^n$. If for some $i$ we have that $\beta_i(x) \leq 0$, then $\beta_j(x) >0$ for all $j=0 \ldots m$ with $j \neq i$. 
\end{assumption}

%
\begin{definition}[\bf{Free space}]\label{def_free_space}
The free space is the set of points $x\in\mathcal{F}\subset\mathbb{R}^n$ where the function $\beta$ in \eqref{eqn_obstacle_function} is nonnegative,
\begin{equation}\label{eqn_free_space}
\mathcal{F} = \left\{x \in \mathbb{R}^n : \beta(x) \geq 0 \right\}.
\end{equation}
\end{definition}

%
Observe that we have assumed that the optimal point $x^*$ is in the interior of free space. We have also assumed that the objective function $f_0$ is strongly convex and twice continuously differentiable and that the same is true of the obstacle functions $\beta_i$. We state these assumptions formally for later reference.

%
\begin{assumption}\label{assum_objective_function} The objective function $f_0$, the obstacle functions $\beta_i$ and the free space $\ccalF$ are such that:
\begin{mylist}
\item[\bf{Optimal point.}]\label{assum_minimum_f0}
 $x^*:=\argmin_x f_0(x)$ is such that $f_0(x^*)\geq 0$ and it is in the interior of the free space, 
 \begin{equation}
 x^* \in \mbox{int}(\mathcal{F}).
 \end{equation}
\item[\bf{Twice differential strongly convex objective}]\label{assum_strong_convexity} The function $f_0$ is twice continuously differentiable and strongly convex in $\mathcal{X}$. The eigenvalues of the Hessian $\nabla^2 f_0(x)$ are therefore contained in the interval $[\lambda_{\min},\lambda_{\max}]$ with $0<\lambda_{\min}$. In particular, strong convexity implies that for all $x,y \in \mathcal{X}$,
\begin{equation}\label{eqn_strong_convexity1}
f_0(y) \geq f_0(x) + \nabla f_0(x)^T(y-x) +\frac{\lambda_{\min}}{2} \| x-y \|^2,
\end{equation}
and, equivalently, 
\begin{equation}\label{eqn_strong_convexity2}
\left( \nabla f_0(y) -\nabla f_0(x) \right)^T(y-x) \geq \lambda_{\min} \| x-y \|^2.
\end{equation}
\item[\bf{Twice differential strongly convex obstacles}]\label{assum_strong_convexity_obstacles} The function $\beta_i$ is twice continuously differentiable and strongly convex in $\mathcal{X}$. The eigenvalues of the Hessian $\nabla^2 \beta_i(x)$ are therefore contained in the interval $[\mu^i_{\min},\mu^i_{\max}]$ with $0<\mu^i{\min}$. 
\end{mylist} 
\end{assumption}

The goal of this paper is to find a navigation function $\varphi$ for the free space $\ccalF$ of the form of Definition \ref{def_free_space} when assumptions \ref{assum_obstacles} and \ref{assum_objective_function} hold. Finding this navigation function is equivalent to attaining the goal in \eqref{eqn_navigation_goal} for almost all initial conditions. We find sufficient conditions for this to be possible when the minimum of the objective function takes the value $f(x^*)=0$. When $f(x^*)\neq0$ we find sufficient conditions to construct a function that satisfies the properties in Definition \ref{def_navigation_function} except for the polar condition that we relax to the function $\varphi$ having its minimum within a predefined distance of the minimum $x^*$ of the potential $f_0$. The construction and conditions are presented in the following section after two pertinent remarks.

%
\begin{remark}[\bf System with dynamics]\label{rmk_dynamics}\normalfont If the system has integrator dynamics, then \eqref{eqn_gradient_flow} can be imposed and problem \eqref{eqn_navigation_goal} be solved by a navigation function. If the system has nontrivial dynamics, a minor modification can be used \cite{koditschek1991control}. Indeed, let $M(x)$ be the inertia matrix of the agent, $g(x,\dot{x})$ and $h(x)$ be fictitious and gravitational forces, and $\tau(x,\dot{x})$ the torque control input. The agent's dynamics can then be written as
\begin{equation}\label{eqn_robot_model}
   M(x) \ddot{x} + g(x,\dot{x}) + h(x) = \tau(x,\dot{x}).
\end{equation}  
The model in \eqref{eqn_robot_model} is of control inputs that generate a torque $\tau(x,\dot{x})$ that acts through the inertia $M(x)$ in the presence of the external forces $g(x,\dot{x})$ and $h(x)$. Let $d(x,\dot{x}) $ be a dissipative field, i.e., satisfying $\dot{x}^T d(x,\dot{x})<0$. Then, by selecting the torque input 
\begin{equation}
   \tau(x,\dot{x}) = -\nabla \varphi(x) + d(x,\dot{x}),
\end{equation}
the behavior of the agent converges asymptotically to solutions of the gradient dynamical system \eqref{eqn_gradient_flow} \cite{koditschek1991control}. In particular, the goal in \eqref{eqn_navigation_goal} is achieved for a system with nontrivial dynamics. Furthermore the torque input above presents a minimal energy solution to the obstacle-avoidance problem \cite{takegaki1981new}.
\end{remark}
 
%
\begin{remark}[\bf Example objective functions]\label{sec_objective_functions}

The attractive potential $f_0(x) = \| x-x^*\|^2$ is commonly used to navigate to  position $x^*$. In this work we are interested in more general potentials that may arise in applications where $x^*$ is unknown a priori. As a first example consider a target location problem in which the location of the target is measured with uncertainty. This results in the determination of a probability distribution $p_{x_0}(x_0)$ for the location $x_0$ of the target. A possible strategy here is to navigate to the expected target position. This can be accomplished if we define the potential
\begin{equation}\label{eqn_example_objective_uncertain_target}
   f_0(x) := \E{\|x-x_0\|} = \int_{\ccalF} \|x-x_0\| \,p_{x_0} (x_0) \, dx_0 
\end{equation}
which is non spherical but convex and differentiable as long as $p_{x_0} (x_0)$ is a nonatomic dsitribution. Alternatives uses of the distribution $p_{x_0} (x_0)$ are possible. An example would be a robust version of \eqref{eqn_example_objective_uncertain_target} in which we navigate to a point that balances the expected proximity to the target with its variance. This can be formulated by the use of the potential $f_0(x) := \E{\|x-x_0\|}+\lambda \var{\|x-x_0\|}$ for some $\lambda>0$.

We can also consider $p$ targets with location uncertainties captured by probability distributions $p_{x_i} (x_i)$ and importance weights $\omega_i$. We can navigate to the expected position of the weighted centroid using the potential
\begin{equation}\label{eqn_example_objective_uncertain_target}
   f_0(x) := \sum_{i=1}^p \omega_i \int_{\ccalF} \|x-x_i\| \,p_{x_i} (x_i) \, dx_i. 
\end{equation}   
Robust formulations of \eqref{eqn_example_objective_uncertain_target} are also possible. 
\end{remark}

%
\section{Navigation Function}\label{sec_navigation_function}
Following the development in \cite{koditschek1990robot} we introduce an order parameter $k>0$ and define the function $\varphi_k$ as
\begin{equation}\label{eqn_navigation_function}
   \varphi_k(x) \triangleq \frac{f_0(x)}{\left(f_0^k(x)+\beta(x) \right)^{1/k}}.
\end{equation}
In this section we state sufficient conditions such that for large enough order parameter $k$, the artificial potential \eqref{eqn_navigation_function} is a navigation function in the sense of Definition \ref{def_navigation_function}. These conditions relate the bounds on the eigenvalues of the Hessian of the objective function $\lambda_{\min}$ and $\lambda_{\max}$ as well as the bounds on the eigenvalues of the Hessian of the obstacle functions $\mu_{\min}^i$ and $\mu_{\max}^i$ with the size of the objects and their distance to the minimum of the objective function $x^*$. The first result concerns the general case where obstacles are defined through general convex functions.
\begin{theorem}\label{theo_general}
Let $\mathcal{F}$ be the free space defined in \eqref{eqn_free_space} satisfying Assumption \ref{assum_obstacles} and let $\varphi_k: \mathcal{F} \rightarrow [0,1]$ be the function defined in \eqref{eqn_navigation_function}. Let $\lambda_{\max}$, $\lambda_{\min}$ and $\mu^i_{\min}$ be the bounds in Assumption \ref{assum_objective_function}. Further let the following condition hold for all $i=1\ldots m$ and for all $x_s$ in the boundary of $\mathcal{O}_i$
\begin{equation}\label{eqn_general_condition}
\frac{\lambda_{\max}}{\lambda_{\min}} \frac{\nabla \beta_i(x_s)^T (x_s-x^*) }{\|x_s - x^* \|^2}  <\mu_{\min}^i.
\end{equation}
Then, for any $\varepsilon>0$ there exists a constant $K(\varepsilon)$ such that if $k>K(\varepsilon)$, the function $\varphi_k$ in \eqref{eqn_navigation_function} is a navigation function with minimum at $\bar{x}$, where $\|\bar{x}-x^*\|<\varepsilon$. Furthermore if $f_0(x^*) =0$ or $\nabla \beta(x^*)=0$, then $\bar{x}=x^*$.
\end{theorem}
\begin{proof}
See Section \ref{sec_proof}.
\end{proof}

%
Theorem \ref{theo_general} establishes sufficient conditions on the obstacles and objective function for which $\varphi_k$ defined in \eqref{eqn_navigation_function} is guaranteed to be a navigation function for sufficiently large order $k$. This implies that an agent that follows the flow \eqref{eqn_gradient_flow} will succeed in navigate towards $x^*$ when $f_0(x^*)=0$. In cases where this is not the case the agent converges to a neighborhood of the $x^*$. This neighborhood can be made arbitrarily small by increasing $k$. Of these conditions \eqref{eqn_general_condition} is the hardest to check and thus the most interesting. Here we make the distinction between verifying the condition in terms of design -- understood as using the result to define which environments can be navigated -- and its verification in operation time. We discuss the first next and we present an algorithm to do the latter in Section \ref{section_varying_k}. Observe that even if it needs to be satisfied at all the points that lie in the boundary of an obstacle, it is not difficult to check numerically in low dimensions. This is because the functions are smooth and thus it is possible to discretize the boundary set with a thin partition to obtain accurate approximations of both sides of \eqref{eqn_general_condition}. In addition, as we explain next, in practice there is no need check the condition on every point of the boundary. Observe first that, generically, \eqref{eqn_general_condition} is easier to satisfy when the ratio $\lambda_{\max}/\lambda_{\min}$ is small and when the minimum eigenvalue $\mu_{\min}^i$ is large. The first condition means that we want the objective to be as close to spherical as possible and the second condition that we don't want the obstacle to be too flat. Further note that the left hand side of \eqref{eqn_general_condition} is negative if $\nabla \beta_i(x_s)$ and $x_s - x^*$ point in opposite directions. This means that the condition can be violated only by points in the border that are ``behind'' the obstacle as seen from the minimum point. For these points the worst possible situation is when the gradient at the border point $x_s$ is aligned with the line that goes from that point to the minimum $x^*$. In that case we want the gradient $\nabla \beta_i(x_s)$ and the ratio $(x_s-x^*)/\|x_s - x^* \|^2$ to be small. The gradient  $\nabla \beta_i(x_s)$ being small with respect to $\mu_{\min}$ means that we don't want the obstacle to have sharp curvature and the ratio $(x_s-x^*)/\|x_s - x^* \|^2$ being small means that we don't want the destination $x^*$ to be too close to the border. In summary, the simplest navigation problems have objectives and obstacles close to spherical and minima that are not close to the border of the obstacles.

The insights described above notwithstanding, a limitation of Theorem \ref{theo_general} is that it does not provide a trivial way to determine if it is possible to build a navigation function with the form in \eqref{eqn_navigation_function} for a given space and objective. In the following section after remarks we consider ellipsoidal obstacles and derive a condition that is easy to check.
%
\begin{remark}[\bf Sufficiently curved worlds \cite{filippidis2012navigation}]\label{rmk_comparisson}\normalfont
  In cases where the objective function is rotationally symmetric for instance $f_0 = \|x-x^*\|^2$ we have that $\lambda_{\max} = \lambda_{min}$. Let $\theta_i$ be the angle between $\nabla \beta_i(x_s)$ and $\nabla f_0(x_s)$, thus \eqref{eqn_general_condition} yields 
  \begin{equation}\label{eqn_our_condition}
 \frac{\|\nabla \beta_i(x_s)\| \cos(\theta_i)}{\|x_s-x^*\|} < \mu_{\min}^i.
    \end{equation}
  For a world to be sufficiently curved there must exist a direction $\hat{t}_i$ such that
  \begin{equation}\label{eqn_their_condition}
 \frac{\|\nabla \beta_i(x_s)\| \cos(\theta_i)\hat{t}_i^TD^2f_0(x_s)\hat{t}_i^T}{\|\nabla f_0(x_s)\|} < \hat{t}_i^T\nabla^2\beta_i(x_s)\hat{t}_i^T.
    \end{equation}
  Since the potential is rotationally symmetric the left hand side of the above equation is equal to the left hand side of \eqref{eqn_our_condition}. Observe that, the right hand side of condition \eqref{eqn_our_condition} is the worst case scenario of the right hand side of condition \eqref{eqn_their_condition}.
  \end{remark}
\begin{remark}
The condition presented in Theorem \ref{theo_general} is sufficient but not necessary. In that sense, and as shown by the numerical example presented before it is possible that the artificial potential is a navigation function even when the condition \eqref{eqn_general_condition} is violated. Furthermore, in the case of spherical potentials it has been show that the artificial potential yields a navigation function for partially non convex obstacles and for obstacles that yield degenerate criticals points \cite{filippidis2012navigation,filippidis2013navigation}. In that sense the problem is not closed and finding necessary conditions for navigation is an open problem. In terms of the objective function it is possible to ensure navigation by assuming local strict convexity at the goal. However under this assumption condition \eqref{eqn_general_condition} takes a form that is not as neat and thus we chose to provide a weaker result in favor of simplicity.
\end{remark}
\subsection{Ellipsoidal obstacles}\label{sec_ellipsoidal_obstacles} 
Here we consider the particular case where the obstacles are ellipsoids. Let $A_i \in \mathcal{M}^{n\times n}$ with $i=1\ldots m$ be $n\times n$ symmetric positive definite matrices and $x_i$ and $r_i$ be the center and the length of the largest axis of each one of the obstacles $\mathcal{O}_i$. Then, for each $i=1\ldots m$ we define $\beta_i(x)$ as
\begin{equation} \label{eqn_beta_i}
\beta_i(x) \triangleq  \left( x-x_i \right)^T A_i \left(x-x_i\right) - \mu_{\min}^i r_i^2,
\end{equation}
The obstacle $\mathcal{O}_i$ is defined as those points in $\mathbb{R}^n$ where $\beta_i(x)$ is not positive. In particular its boundary, $\beta_i(x) =0$, defines an ellipsoid whose largest axis has length $r_i$
\begin{equation}\label{eqn_ellipses_boundary}
\frac{1}{\mu_{\min}^i}\left( x-x_i \right)^T A_i \left(x-x_i\right) =  r_i^2.
\end{equation} 
For the particular geometry of the obstacles considered in this section, Theorem \ref{theo_general} takes the following simplified form.
\begin{theorem}\label{theo_ellipses}
Let $\mathcal{F}$ be the free space defined in \eqref{eqn_free_space} satisfying Assumption \ref{assum_obstacles}, and $\varphi_k: \mathcal{F} \rightarrow [0,1]$ be the function defined in \eqref{eqn_navigation_function}.  Let $\lambda_{\max}$, $\lambda_{\min}$, $\mu_{\max}^i$ and $\mu_{\min}^i$ be the bounds from Assumption \ref{assum_objective_function}. Assume that $\beta_i$ takes the form of \eqref{eqn_beta_i} and the following inequality holds for all $i=1..m$
\begin{equation}\label{eqn_condition_ellipses}
\frac{\lambda_{\max}}{\lambda_{\min}} \frac{\mu_{\max}^i}{ \mu_{\min}^i} < 1 +\frac{d_i}{r_i},
\end{equation}
where $d_i\triangleq\|x_i -x^* \|$ . Then, for any $\varepsilon>0$ there exists a constant $K(\varepsilon)$ such that if $k>K(\varepsilon)$, the function $\varphi_k$ in \eqref{eqn_navigation_function} is a navigation function with minimum at $\bar{x}$, where $\|\bar{x}-x^*\|<\varepsilon$. Furthermore if $f_0(x^*) =0$ or $\nabla \beta(x^*)=0$, then $\bar{x}=x^*$.  
\end{theorem}
\begin{proof}
See Appendix \ref{ap_upper_bound_ellipses_proof}.
\end{proof}
Condition \eqref{eqn_condition_ellipses} gives a simple form of distinguishing from spaces with ellipsoidal obstacles in which it is possible to build a Koditscheck-Rimon navigation function and spaces in which this is not possible. The more eccentric the obstacles and the level sets of the objective function are, the larger becomes the left hand side of \eqref{eqn_condition_ellipses}. In particular, for a flat obstacle -- understood as an ellipses having its minimum eigenvalue equal to zero-- the considered condition is impossible to satisfy. Notice that this is consistent with Theorem \ref{theo_general}. On the other hand, the proximity of obstacles plays a role. By increasing the distance between the center of the obstacles and the objective, $d_i$ -- or by decreasing the size of the obstacles, $r_i$ -- we decrease the proximity of obstacles in the space, thus increasing the ratio of the right hand side of \eqref{eqn_condition_ellipses} and therefore making simplifying the navigation in the environment. 

A question that remains unanswered is whether the inequality \eqref{eqn_condition_ellipses} is tight or not. Let us point out that both conditions \eqref{eqn_general_condition} and \eqref{eqn_condition_ellipses} are shown to be sufficient but not necessary. In that sense, when the conditions are violated it is possible to build a world in which the proposed artificial potential is a navigation function. The notion of tightness that we discuss next corresponds to the idea that if the condition is violated then an example where the artificial potential defined in \eqref{eqn_navigation_function} fails to be a navigation function. Let $v_{\min}$ be the eigenvector associated to the eigenvalue $\lambda_{\min}$. For any situation in which $v_{\min}$ is aligned with the direction $x_i -x^*$, then if condition \eqref{eqn_condition_ellipses} is violated with equality the artificial potential in \eqref{eqn_navigation_function} fails to be a navigation function and in that sense condition \eqref{eqn_condition_ellipses} is tight. This is no longer the case if these directions are not aligned. To illustrate the above discussion we present consider the following example in $\mathbb{R}^2$ with only one circular obstacle of radius $2$ and objective function given by 
\begin{equation}\label{eqn_objective_for_example}
f_0(x) = x^T \left(\begin{array}{cc}
1 &0 \\
0 &\lambda_{\max}
\end{array}\right)x,
\end{equation}
For this example, the minimum of the objective function is attained at the origin and the left hand side of \eqref{eqn_condition_ellipses} takes value $\lambda_{\max}$. For the first two simulations we consider the case in which the direction $x_i -x^*$ is aligned with the direction of the eigenvector associated with the smallest eigenvalue of the objective function. With this purpose we place the center of the obstacle in the horizontal axis at $(-4,0)$. The right hand side of \eqref{eqn_condition_ellipses} takes therefore the value $3$. In the simulation depicted in Figure \ref{fig_trajectory_with_minimum}, $\lambda_{\max}$ is set to be three, therefore violating condition \eqref{eqn_condition_ellipses}. As it can be observed a local minimum other than $x^*$ appears to the left of the obstacle to which the trajectory converges. Thus, the potential defined in \eqref{eqn_navigation_function} fails to be a navigation function.
In Figure \ref{fig_trajectory_obstacle_not_alligned} we observe an example in which the trajectory converges to $x^*$ and condition \eqref{eqn_condition_ellipses} is violated at the same time. Here, the center of the obstacle is placed at $(0,-4)$, and therefore the direction $x_i-x^*$ is no longer aligned with the eigenvector of the Hessian of the objective function associated to the minimum eigenvalue. Hence showing that condition \eqref{eqn_condition_ellipses} is loose when those directions are not collinear.

\begin{figure}
\centering
\includegraphics[width=0.5\textwidth]{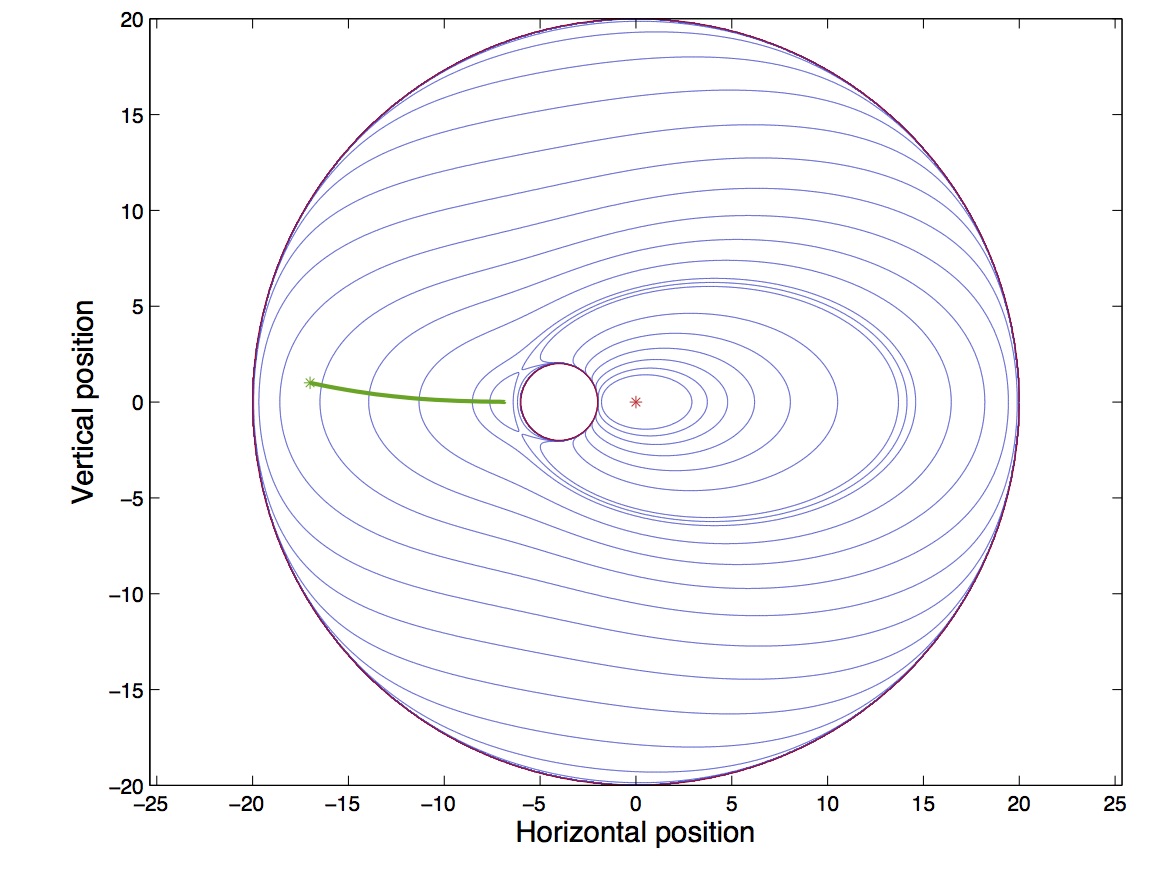}
\caption{ Function $\varphi_k$ fails to be a navigation function when the left and right hand sides of \eqref{eqn_condition_ellipses} are equal. Observe the presence of a local minimum different from the minimum of $f_0$ to which the trajectory is attracted. The experiment was performed with $k=10$.}
\label{fig_trajectory_with_minimum}
\end{figure}
%
%
\begin{figure}
\centering
\includegraphics[width=0.5\textwidth]{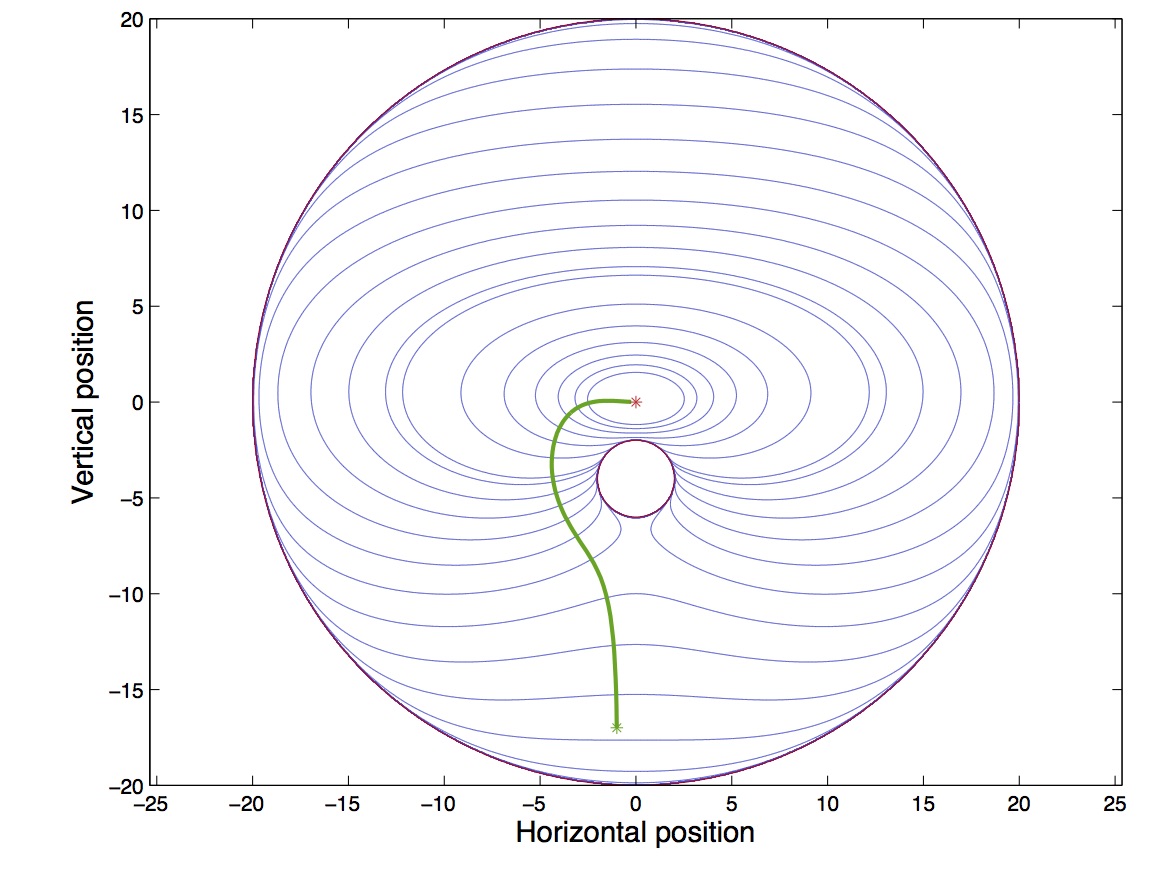}
\caption{Condition \eqref{eqn_condition_ellipses} is violated, however $\varphi_k$ is a navigation function. We do not observe a local minimum on the side of the obstacle that is opposed to the minimum of $f_0$ as we do in Figure \ref{fig_trajectory_with_minimum}. The latter is because the direction given by the center of the obstacle and $x^*$ is not aligned with the direction corresponding to the maximum eigenvalue of the Hessian of $f_0$.  }
\label{fig_trajectory_obstacle_not_alligned}
\end{figure}
Notice that the problem of navigating a spherical world to reach a desired destination $x^*$ \cite{koditschek1990robot} can be understood as particular case where the objective function takes the form $\|x - x^* \|^2$  and the obstacles are spheres. In this case $\varphi_k$ is a navigation function for some large enough $k$ for every valid world (satisfying Assumption \ref{assum_obstacles}), irrespectively of the size and placement of the obstacles. This result can be derived as a corollary of Theorem \ref{theo_ellipses} by showing that condition \eqref{eqn_condition_ellipses} is always satisfied in the setting of \cite{koditschek1990robot}.

\begin{corollary}
Let $\mathcal{F}\subset E^n$ be the set defined in  \eqref{eqn_free_space} and let $\varphi_k: \mathcal{F} \rightarrow [0,1]$ be the function defined in \eqref{eqn_navigation_function}. Let $\mathcal{F}$ verify Assumption \ref{assum_obstacles} and let $f_0(x) = \|x-x^* \|^2$. Let the obstacles be hyper spheres of centers $x_i$ and radii $r_i$ for all $i=1..m$.
Then there exists a constant $K$ such that if $k$ in \eqref{eqn_navigation_function} is larger than $K$, then $\varphi_k$ is a navigation function. 
\end{corollary}
\begin{proof}
Since spherical obstacles are a particular case of ellipsoids the hypothesis of Theorem \ref{theo_ellipses} are satisfied. To show that $\varphi_k$ is a navigation function we need to show that condition \eqref{eqn_condition_ellipses} is satisfied. For this obstacle geometry we have $\mu_{\min}^i = \mu_{\max}^i$ for all $i=1\ldots m$. On the other hand, the Hessian of the function $f_0(x) = \| x- x^*\|^2$ is given by $\nabla^2 f_0(x) = 2 I$, where $I$ is the $n\times n$ identity matrix. Thus, all its eigenvalues are equal. This implies that the left hand side of \eqref{eqn_condition_ellipses} takes the value one. On the other hand, since $d_i$ and $r_i$ are positive quantities the right hand side of \eqref{eqn_condition_ellipses} is strictly larger than one. Hence the condition is always satisfied and therefore $\varphi_k(x)$ is a navigation function for some large enough $k$.
\end{proof}
%
%
\section{Proof of Theorem \ref{theo_general}}\label{sec_proof}
In this section we show that $\varphi_k$, defined in \eqref{eqn_navigation_function} is a navigation function under the hypotheses of Theorem \ref{theo_general} by showing that it satisfies Definition \ref{def_navigation_function}.
%
\subsection{Twice Differentiability and Admissibility}
The following lemma shows that the artificial potential \eqref{eqn_navigation_function} is twice continuously differentiable and admissible.  
\begin{lemma}[\bf{Differentiability and admissibility}]\label{lemma_smooth}
Let $\mathcal{F}$ be the set defined in \eqref{eqn_free_space} and let $\varphi_k : \mathcal{F} \to [0,1]$ be the function defined in \eqref{eqn_navigation_function}. Then, $\varphi_k$ is admissible and twice continuously differentiable on $\mathcal{F}$. 
\end{lemma}
\begin{proof}
  Let us show first that $\varphi_k$ is twice continuously differentiable. To that end we first show that the denominator of \eqref{eqn_navigation_function} is strictly positive. For any $x\in \mbox{int}(\ccalF)$ it holds that $\beta(x)>0$ (c.f. \eqref{eqn_free_space}). Hence  $f_0^k(x) +\beta(x) >0$ because $f_0$ is nonnegative (c.f. Assumption \ref{assum_minimum_f0}). The same holds for $x\in\partial\ccalF$ because the minimum of $f_0$ is not in $\partial\ccalF$ (c.f. Assumption \ref{assum_minimum_f0}).     
Therefore $\left(f_0^k(x) + \beta(x)\right)^{-1/k}$ is twice continuously differentiable in the free space since $f_0$ and $\beta$ are twice continuously differentiable (c.f Assumption \ref{assum_minimum_f0}). Hence $\varphi_k$ is twice continuously differentiable since it is the product of twice continuously differentiable functions. To show admissibility observe that on one hand for every $x\in \mbox{int}(\mathcal{F})$ we have that $\beta(x)>0$, thus $\varphi_k(x)<1$. On the other hand, if $x\in \partial\mathcal{F}$ we have that $\beta(x)=0$, hence $\varphi_k(x) = 1$. Thus, the pre image of $1$ by $\varphi_k$ is the boundary of the free space. This completes the proof.
\end{proof}
%
%
\subsection{The Koditschek-Rimon potential $\varphi_k$ is polar on $\mathcal{F}$}\label{sec_polar}
In this section we show that the function $\varphi_k$ defined in \eqref{eqn_navigation_function} is polar on the free space $\mathcal{F}$ defined in \eqref{eqn_free_space}. Furthermore we show that if $f_0(x^*)=0$ or if $\nabla \beta(x^*)=0$, then its minimum coincides with the minimum of $f_0$. If this is not the case, then the minimum of $\varphi_k(x)$ can be placed arbitrarily close to $x^*$ by increasing the order parameter $k$.
In what follows it is convenient to define the product of all the obstacle functions except $\beta_i$
\begin{equation}\label{eqn_obstacle_complement}
\bar{\beta}_i (x) \triangleq \prod_{j=0,  j\neq i}^m \beta_j(x).
\end{equation}
Then, for any $i=0 \ldots m$, the gradient of the obstacle function can be written as
\begin{equation}\label{eqn_nabla_obstacle}
\nabla \beta(x) = \beta_i(x) \nabla \bar{\beta}_i(x) + \bar{\beta}_i (x) \nabla \beta_i(x).
\end{equation}
The next lemma establishes that $\varphi_k(x)$ does not have critical points in the boundary of the free space.
\begin{lemma}\label{lemma_critical_points_interior}
Let $\mathcal{F}$ be the set defined in \eqref{eqn_free_space} satisfying Assumption \ref{assum_obstacles} and let $\varphi_k: \mathcal{F} \rightarrow [0,1]$ be the function defined in \eqref{eqn_navigation_function}. Then if Assumption \ref{assum_objective_function} holds there are not critical points of $\varphi_k$ in the boundary of the free space. 
\end{lemma}
\begin{proof}
For any $x\in\mathcal{F}$ the gradient of $\varphi_k$ is given by
\begin{equation}\label{eqn_nabla_phi}
\begin{split}
\nabla \varphi_k(x) &= \left(f_0^k(x)+\beta(x)\right)^{-1 -\frac{1}{k}}\\
&\left(\beta(x)\nabla f_0(x)-\frac{f_0(x)\nabla \beta(x)}{k} \right).
\end{split}
\end{equation}
In particular, if $x\in\partial\ccalF$ we have that $\beta(x)=0$ (c.f. \eqref{eqn_free_space}) and the above expression reduces to 
\begin{equation}
\nabla \varphi_k(x) = -\frac{f_0^{-k}(x)}{k}\nabla \beta(x).
\end{equation}
Since $f_0$ is nonnegative and its minimum is not in the boundary of the free space (c.f Assumption \ref{assum_objective_function}), it must be the case that $f_0(x)>0$. It is left to show that $ \nabla \beta(x) \neq 0$ for all $x\in \partial \mathcal{F}$.
In virtue of Assumption \ref{assum_obstacles} the obstacles do not intersect. Hence if $x\in\partial\mathcal{F}$, it must be the case that for exactly one of the indices $i =0\ldots m$ we have that $\beta_i(x)=0$ (c.f. \eqref{eqn_obstacle_function}). Denote by $i^*$ this particular index. Then \eqref{eqn_nabla_obstacle} reduces to 
\begin{equation}
\nabla \beta(x) =  \bar{\beta}_{i^*} (x) \nabla \beta_{i^*}(x). 
\end{equation}
Furthermore we have that for all $j \neq i^*$, $\beta_j(x) >0$ (c.f. \eqref{eqn_beta_i}) hence $\bar{\beta}(x)_{i^*} >0$. Since the obstacles are non empty open sets and in its boundary $\beta_{i^*}(x) =0$ and in its interior $\beta_{i^*}<0$, because $\beta_{i^*}$ is convex it must be the case that $\nabla \beta_{i^*}(x) \neq 0$ for any $x\in \partial \mathcal{O}_{i^*}$. An analogous argument holds for the case of $\beta_0$. This shows that $\nabla \beta(x) \neq 0$ and therefore, there are no critical points in the boundary of the free space.
\end{proof}
%
In the previous lemma we showed that there are not critical points at the boundary of $\varphi_k(x)$, however we show next that these are either placed  arbitrarily close to the boundary of the free space or to $x^*$. We formalize this result next.  
%
%
\begin{lemma}\label{lemma_critical_points_new}
  Let $\mathcal{F}$ be the free space defined in \eqref{eqn_free_space} satisfying Assumption \ref{assum_obstacles} and let $\varphi_k: \mathcal{F} \rightarrow [0,1]$ be the function defined in \eqref{eqn_navigation_function}. Then $\varphi_k(x)$ has critical points $x_c\in\mbox{int}(\ccalF)$ for all $k>0$ and there exists $\varepsilon_0>0$ such that for and any $\varepsilon\in(0,\varepsilon_0]$ there exits $K_0(\varepsilon)>0$ such that if $k>K_0(\varepsilon)$ either $\|\nabla f_0(x_c)\|<\varepsilon$ or $\|\beta_i(x_c)\|<\varepsilon$ for exactly one $i=1\ldots m$.
\end{lemma}
\begin{proof}
  See appendix \ref{ap_lemma_critical_points_new}.
\end{proof}
The previous lemma shows that the critical points of the navigation function can be pushed arbitrarily close to the boundary of one of the obstacles or arbitrarily close to the minimum of the objective function by selecting $k$ sufficiently large. In the next Lemma we show that for large enough $k$ the critical points close to the boundary of the obstacles cannot be local minima. The following lemma as well as Lemma \ref{lemma_non_degeneracy} can be derived from \cite{filippidis2012navigation, filippidis2013navigation, filippidis2011navigation}. We report the proofs since they are shorter for the particular class of obstacles here considered.

%
\begin{lemma}\label{lemma_saddle_points}
Let $\mathcal{F}$ be the free space defined in \eqref{eqn_free_space} satisfying Assumption \ref{assum_obstacles} and let $\varphi_k: \mathcal{F} \rightarrow [0,1]$ be the function defined in \eqref{eqn_navigation_function}. Let $\lambda_{\max}$, $\lambda_{\min}$ and $\mu^i_{\min}$ the bounds in Assumption \ref{assum_objective_function}. Further let \eqref{eqn_general_condition}\ hold for all $i=1\ldots m$ and for any $x\in\partial\mathcal{O}_i$. Then, there exists $\varepsilon_1>0$ such that for any $\varepsilon\in(0,\varepsilon_1]$, there exists $K_1(\varepsilon)$ such that if $k>K_1(\varepsilon)$, no critical point $x_c$ such that $\beta_i(x_c)<\varepsilon$ is a local minimum.
\end{lemma}
\begin{proof}
See Appendix \ref{ap_proof_saddles}.
\end{proof}
In the previous Lemma we established that the critical points near the boundary of the free space are not local minima. Therefore the critical points close to $x^*$ have to be. In the next Lemma we formalize this result and we show that for large enough $k$ there is only one non degenerate critical point.
%
%
\begin{lemma}\label{lemma_polar_general}
  Let $\mathcal{F}$ be the free space defined in \eqref{eqn_free_space} satisfying Assumption \ref{assum_obstacles} and let $\varphi_k: \mathcal{F} \rightarrow [0,1]$ be the function defined in \eqref{eqn_navigation_function}. Let $\lambda_{\max}$, $\lambda_{\min}$ and $\mu^i_{\min}$ the bounds in Assumption \ref{assum_objective_function}. Further let \eqref{eqn_general_condition} hold for all $i=1\ldots m$ and for all $x_s$ in the boundary of $\mathcal{O}_i$. Then, for any $\varepsilon\in(0,\varepsilon_1]$ there exists $K_2(\varepsilon)>0$ such that if $k>K_2(\varepsilon)$, $\varphi_k$ is polar with minimum $\bar{x}$ such that $\|\bar{x}-x^*\|<\varepsilon$. Moreover if $f_0(x^*)=0$ or $\nabla \beta(x^*)=0$, then $\bar{x}=x^*$.
\end{lemma}
\begin{proof}
  See Appendix \ref{ap_general_proof_minimum}.
\end{proof}
The previous lemma establishes that $\varphi_k$ is polar, with its minimum arbitrarily close to $x^*$ hence we are left to show that the $\varphi_k(x)$ is Morse which we do next.
%
%
\subsection{Non degeneracy of the critical points}\label{sec_nondegeneracy}
In the previous section, we showed that the navigation function is polar and that the minimum is non degenerate. Hence, to complete the proof we need to show that the critical points close to the boundary are not degenerate. We formalize this in the following lemma. 
%
%
\begin{lemma}\label{lemma_non_degeneracy}
Let $\mathcal{F}$ be the free space defined in \eqref{eqn_free_space} satisfying Assumption \ref{assum_obstacles} and let $\varphi_k: \mathcal{F} \rightarrow [0,1]$ be the function defined in \eqref{eqn_navigation_function}. Let $\lambda_{\max}$, $\lambda_{\min}$ and $\mu^i_{\min}$ the bounds in Assumption \ref{assum_objective_function}. Further let \eqref{eqn_general_condition} hold for all $i=1\ldots m$ and for all points in the boundary of $\mathcal{O}_i$. Then, for any $\varepsilon\in(0,\varepsilon_0)$ there exists $K_3(\varepsilon)$ such that if $k>K_3(\varepsilon)$ the critical points $x_s$ of $\varphi_k$ satisfying $\beta_i(x_s)<\varepsilon$ for $i=1\ldots m$ are non degenerate.
\end{lemma}
\begin{proof}
  We showed in \ref{lemma_saddle_points} that the Hessian of $\varphi_k$ evaluated at the critical points satisfying $\beta_i(x_s)<\varepsilon<\varepsilon_0$ has $n-1$ negative eigenvalues when $k>K_1(\varepsilon)$. In particular the subspace of negative eigenvalues is the plane normal to $\nabla \beta(x_s)$.  Hence, to show that $\varphi_k$ is Morse it remains to be shown that the quadratic form associated to $\nabla^2 \varphi_k$ at the critical points close to the boundary of the subspace is positive when evaluated in the direction of $v = \nabla \beta(x_s)/\|\nabla \beta(x_s)\|$. As previously argued $v^T\nabla \varphi_k(x_s)v>0$ if and only if
\begin{equation}
\begin{split}
v^T\left(\beta(x_s)\nabla^2 f_0(x_s) + (1-\frac{1}{k})\nabla \beta(x_s) \nabla f_0^T(x_s) \phantom{\frac{1}{2}}\right.\\
\left.\phantom{\frac{1}{2}}-\frac{f_0(x_s)}{k}\nabla^2\beta(x_s)\right)v>0. 
\end{split}
\end{equation}
Note that $\beta(x_s)v^T \nabla ^2 f_0(x_s)v$ is positive since $f_0$ is convex (c.f. Assumption \ref{assum_objective_function}) and $\beta(x)\geq 0$ for all $x\in\ccalF$ (c.f. \eqref{eqn_free_space}).

For any $k>1$ the second term in the above equation is positive since $\nabla f_0(x_s)$ and $\nabla \beta(x_s)$ point in the same direction. Moreover since at the boundary of the obstacle $\nabla \beta(x) \neq 0$ (see Lemma \ref{lemma_critical_points_interior}), for any $\delta>0$, there exists $K_3\prime(\delta)$ such that if $k>K_3(\delta)$, then $\|\nabla \beta(x_s)\|>\delta$. By virtue of Lemma \ref{lemma_critical_points_new} $\| f_0(x_s)\|>\varepsilon_0$ hence the second term in the above equation is bounded away from zeros by a constant independent of $k$. Finally since $f_0$ and $\beta$ are twice continuously differentiable $f_0(x)\nabla^2 \beta(x)$ is bounded by a constant independent of $k$ for all $x\in\ccalF$. Hence there exists $K_3(\varepsilon)>0$ such that if $k>K_3(\varepsilon)$ the above inequality holds and therefore the critical points are non degenerate.
\end{proof}
To complete the proof of Theorem \ref{theo_general} it suffice to choose $K = \max\{K_2(\varepsilon),K_3(\varepsilon)\}$.
%
%


%
\section{Practical considerations}\label{sec_navigation_function_unknown}

The gradient controller in \eqref{eqn_gradient_flow} utilizing the navigation function $\varphi=\varphi_k$ in \eqref{eqn_navigation_function} succeeds in reaching a point arbitrarily close to the minimum $x^*$ under the conditions of Theorem \ref{theo_general} or Theorem \ref{theo_ellipses}. However, the controller is not strictly local because constructing $\varphi_k$ requires knowledge of all the obstacles. This limitation can be remedied by noting that the encoding of the obstacles is through the function $\beta(x)$ which is defined by the product of the functions $\beta_i(x)$ [c.f. \eqref{eqn_obstacle_function}]. We can then modify $\beta(x)$ to include only the obstacles that have already been visited. Consider a given constant $c>0$ related to the range of the sensors measuring the obstacles and define the $c$-neighborhood of obstacle $\ccalO_i$ as the set of points with $\beta_i(x) \leq c$. For given time $t$, we define the set of obstacles of which the agent is aware as the set of obstacles of which the agent has visited their $c$-neighborhood at some time $s\in [0,t]$,
\begin{equation}\label{eqn_awareness_set}
   \ccalA_c(t) \triangleq 
       \Big\{  i : \beta_i(x(s)) \leq c,  \text{\ for some\ } s\in[0,t] \Big\}.
\end{equation}
The above set can be used to construct a modified version of $\beta(x)$ that includes only the obstacles visited by the agent,
\begin{equation}\label{eqn_partial_obstacle_function}
   \beta_{\ccalA_c(t)}(x) \triangleq \beta_0(x) \prod_{i \in \ccalA_c(t)} \beta_i(x).
\end{equation}
Observe that the above function has a dependance on time through the set $\ccalA_c(t)$ however this dependence is not explicit as the set is only modified when the agent reaches the neighborhood of a new obstacle. In that sense $\ccalA_c(t)$ behaves as a switch depending only of the position. Proceeding by analogy to \eqref{eqn_navigation_function} we use the function $\beta_{\ccalA_c(t)}(x)$ in \eqref{eqn_partial_obstacle_function} to define the switched potential $\varphi_{k,\ccalA_c(t)}(x) : \ccalF_{\ccalA_c(t)} \to \mathbb{R}$ taking values
\begin{equation}\label{eqn_partial_navigation_function}
\varphi_{k,\ccalA_c(t)}(x)  \triangleq \frac{f_0(x)}{\left(f_0^k(x)+\beta_{\ccalA_c(t)}(x) \right)^{1/k}}.
\end{equation}
The free space $\ccalF_{\ccalA_c(t)}$ is defined as in \eqref{eqn_free_space_set}, with the difference that we remove only those obstacles for which $i \in \ccalA_c(t)$. Observe that $\ccalF_{\ccalA_c(t)} \subseteq \ccalF_{\ccalA_c(s)}$ if $t>s$. We use this potential to navigate the free space $\ccalF$ according to the switched controller
\begin{equation}\label{eqn_partial_gradient_flow}
   \dot{x} = - \nabla \varphi_{k,\ccalA_c(t)}(x) .
\end{equation}
Given that $\varphi_{k,\ccalA_c(t)}(x)$ is a switched potential, it has points of discontinuity. The switched gradient controller in \eqref{eqn_partial_gradient_flow} is interpreted as following the left limit at the discontinuities. The solution of system \eqref{eqn_partial_gradient_flow} converges to the minimum of $f_0(x)$ while avoiding the obstacles for a set of initial conditions whose measure is one as we formally state next. 
\begin{theorem}\label{theo_switched}
Let $\mathcal{F}$ be the free space defined in \eqref{eqn_free_space} verifying Assumption \ref{assum_obstacles} and let $\ccalA_c(t)$ for any $c>0$ be the set defined in \eqref{eqn_awareness_set}. Consider the switched navigation function  
$\varphi_{k,\ccalA_c(t)}: \ccalF_{\ccalA_c(t)}  \rightarrow [0,1]$ to be the function defined in \eqref{eqn_partial_navigation_function}. Further let condition \eqref{eqn_general_condition} hold for all $i=1\ldots m$ and for all $x_s$ in the boundary of $\mathcal{O}_i$.
Then, for any $\varepsilon>0$ there exists a constant $K(\varepsilon)$ such that if $k>K(\varepsilon)$, for a set of initial conditions of measure one, the solution of the dynamical system \eqref{eqn_partial_gradient_flow} verifies that $x(t) \in \mathcal{F}$ for all $t\in[0,\infty)$ and its limit is $\bar{x}$, where $\|\bar{x}-x^*\|<\varepsilon$. Moreover if $f_0(x^*) =0$ or $\nabla \beta(x^*)=0$, then $\bar{x}=x^*$.
\end{theorem}
\begin{proof}
See Appendix \ref{ap_theo_switched_proof}.
\end{proof}
Theorem \ref{theo_switched} shows that it is possible to navigate the free space $\ccalF$ and converge asymptotically to the minimum of the objective function $f_0(x)$ by implementing the switched dynamical system \eqref{eqn_partial_gradient_flow}. This dynamical system only uses information about the obstacles that the agent has already visited. Therefore, the controller in \eqref{eqn_partial_gradient_flow} is a spatially local algorithm because the free space is not known a priori but observed as the agent navigates. Do notice that the observation of the obstacles is not entirely local because their complete shape is assumed to become known when the agent visits their respective $c$-neighborhoods. Incremental discovery of obstacles is also consider in \cite{filippidis2011adjustable} for the case of spherical worlds and the proof are similar to that of Theorem \ref{theo_switched}. We also point out that a minor modification of \eqref{eqn_partial_gradient_flow} can be used for systems with dynamics as we formalize in the next proposition.  
%
\begin{corollary}
Consider the system given by \eqref{eqn_robot_model}. Let $\varphi_{k,\ccalA_c(t)}(x)$ be the function given by \eqref{eqn_partial_navigation_function} and let $d(x,\dot{x})$ be a dissipative field, then by selecting the torque input 
\begin{equation}\label{eqn_proposition}
   \tau(x,\dot{x}) = -\nabla \varphi_{k,\ccalA_c(t)}(x) + d(x,\dot{x}),
\end{equation}
the behavior of the agent converges asymptotically to solutions of the gradient dynamical system \eqref{eqn_partial_gradient_flow}. 
\end{corollary}
\begin{proof}
From the proof of Theorem \ref{theo_switched} it follows that there exists a finite time $T>0$ such that $\ccalA_c(t)$ is constant for any $t\geq T$ [cf.\eqref{eqn_finite_time}]. Then for any $t \geq T$ the dynamical system given by \eqref{eqn_robot_model} with the torque input \eqref{eqn_proposition} is equivalent to the system discussed in Remark \ref{rmk_dynamics} and the proof of \cite{koditschek1991control} follows.
\end{proof}
The above corollary shows that the goal in \eqref{eqn_navigation_goal} is achieved for a system with nontrivial dynamics when the obstacles are observed in real time. Observe that Theorems \ref{theo_general}, \ref{theo_ellipses} and \ref{theo_switched} provides guarantees that a tuning parameter $K$ exists, yet they do not provide a way to determine it in operation time. We devote the next section to overcome this limitation. 

%
\subsection{Adjustable $k$}\label{section_varying_k}
In this section we discuss present an algorithm that allows the agent to converge to minimum of the artificial potential in finite time if condition \eqref{eqn_general_condition} is satisfied. In cases where the condition is not satisfied then the robot will increase the tuning parameter $k$ until the value reaches the maximum number that he can handle. In that sense the algorithm in this section provides a way of checking the condition in operation time. Let $f(x)$ be a vector field and define its normalization as 
\begin{equation}
  \overline{f(x)} = \left\{\begin{array}{c}
  \frac{f(x)}{\|f(x)\|} \quad \mbox{if} \quad f(x) \neq 0 \\
  0 \quad \mbox{if} \quad f(x) =0
  \end{array}\right.
  \end{equation}
The main reason for considering normalized gradient flows is that the convergence speed is not affected by an increment of the parameter $k$ as happens when considering gradient descent. Furthermore convergence to critical points happens in finite (c.f. Corollary \cite{cortes2006finite}). The latter is key to ensure convergence to the minimum of the navigation function by Algorithm \ref{alg_adjustable_k} in finite time or to reach the maximum value $K_{max}$ that the robot can handle. We formalize this result next. 
%
%
{\small \begin{algorithm}[t] 
\caption{Adjustable $k$}
{\small \label{alg_adjustable_k} 
\begin{algorithmic}[1]
\Require initialization $x_{0}$ and $k_{0}$, 
\State $k\gets k_0$
\State Normalized Gradient descent $\dot{x} = -\overline{\nabla \varphi_k(x)}.$
\State $x_0 \gets x(t_{f})$
\While {$x_0 \neq \argmin \varphi_k(x)$ or $k<K_{max}$}
\State $k\gets k+1$
\State Draw random point $x_{rand}$ in neighborhood of $x_0$
\State Move towards $x_{rand}$ by doing $\dot{x} = -\overline{(x-x_{rand})}.$
\State $x_0 \gets x_{rand}$
\State Normalized Gradient descent $\dot{x} = -\overline{\nabla \varphi_k(x)}.$
\EndWhile
\end{algorithmic}}
\end{algorithm}}
%
\begin{theorem}
  Let $\mathcal{F}$ be the free space defined in \eqref{eqn_free_space} satisfying Assumption \ref{assum_obstacles} and let $\varphi_k: \mathcal{F} \rightarrow [0,1]$ be the function defined in \eqref{eqn_navigation_function}. Let $\lambda_{\max}$, $\lambda_{\min}$ and $\mu^i_{\min}$ be the bounds in Assumption \ref{assum_objective_function}. Further let \eqref{eqn_general_condition} hold for all $i=1\ldots m$ and for all $x_s$ in the boundary of $\mathcal{O}_i$. Then Algorithm \ref{alg_adjustable_k} finishes in finite time and either $k=K_{max}$ or the final position $x_f$ is such that $x_f= \argmin \varphi_k(x)$. 
\end{theorem}
\begin{proof}
  Let us assume that the maximum $K_{max}$ that the robot can handle is smaller than the $K$ of \ref{theo_general}. Notice that the normalized gradient flow convergences to a critical point of the artificial potential $\varphi_k(x)$ (c.f. Corollary 10 \cite{cortes2006finite}). If the initial parameter $k_0$ is large enough then this implies convergence to the minimum and the algorithm finishes in finite time. If not then it means that the robot is in a local minimum. The convergence to the randomly selected point also happens in finite time because of the result in \cite{cortes2006finite}. Therefore in finite time it will be the case that $k>K$. This being the case implies that $\varphi_k(x)$ is a navigation function and therefore the critical points next to the obstacles are saddles. Also, with probability one the point $x_{rand}$ will be in the unstable manifold of the saddle. Which ensures that after one more normalized gradient descent the agent converges to the minimum of the navigation function. On the other hand if $K>K_{max}$ then the agent does not converge to minimum of the navigation function but the algorithm terminates in finite time due to the fact that the gradients flow converge in finite time. 
\end{proof}
The previous theorem shows that the adaptive strategy for selecting the parameter $k$ finishes in finite time either by converging to the minimum of the navigation function or by reaching the maximum $K$ that the agent can handle. If the second happens then it must be the case that the problem that the agent is trying to satisfy if such that it violates condition \eqref{eqn_general_condition} and in that sense Algorithm \ref{alg_adjustable_k} allows to identify where the does not hold. Observe that to avoid jittering around the critical points it is possible to stop the gradient flows when the norm of the gradient is smaller than a given tolerance.

%
\section{Numerical experiments}\label{sec_numerical_examples}
We evaluate the performance of the navigation function \eqref{eqn_partial_navigation_function} in different scenarios. To do so, we consider a discrete approximation of the gradient flow \eqref{eqn_partial_gradient_flow}
\begin{equation}\label{eqn_approx_flow}
x_{t+1} = x_t - \varepsilon_t \nabla \varphi_{k,\ccalA_c(t)}(x_t).
\end{equation}
Where $x_0$ is selected at random and $\varepsilon_t$ is a diminishing step size. In Section \ref{sec_numerical_ellipses} we consider a free space where the obstacles considered are ellipsoids --the obstacle functions $\beta_i(x)$ for $i=1\ldots m$ take the form \eqref{eqn_beta_i}. In particular we study the effect of diminishing the distance between the obstacles while keeping the length of its mayor axis constant. In this section we build the free space such that condition \eqref{eqn_condition_ellipses} is satisfied. As already shown through a numerical experiment in Section \ref{sec_navigation_function} the previous condition is tight for particular configurations, yet we observed that navigation is still possible if \eqref{eqn_condition_ellipses} is violated (c.f. Figure \ref{fig_trajectory_obstacle_not_alligned}). This observation motivates the study in Section  \ref{sec_condition_violated} where we consider worlds were \eqref{eqn_condition_ellipses} is violated. In \ref{sec_general_convex_obstacle} we consider egg shaped obstacles as an example of convex obstacles other than ellipsoids. The numerical section concludes in Section \ref{sec_dynamics} and \ref{sec_wheeled} where we consider respectively a system with double integrator dynamics and a wheeled robot.
%
\subsection{Elliptical obstacles in $\mathbb{R}^2$ and $\mathbb{R}^3$}\label{sec_numerical_ellipses}
%
 \begin{table*}
   \centering
        \begin{subtable}[t]{0.4\linewidth}
\begin{tabular}{ c | c| c |c| c }
$d$ & $k$ & max final dist & min initial dist & collisions\\
\hline
  10 & 2& $4.45 \times 10^{-2}$ & 10.06&0 \\
   9 & 2 &$17.25$ & 10.01 &0\\
   9 & 5 & $4.45 \times 10^{-2}$ & 10.01 &0\\
   6 & 5 & $21.61$ & 10.01 &0\\
   6 & 7 & $4.74 \times 10^{-2}$ & 10.02 &0\\
   5 & 7 & $22.29$ & 10.027 &0\\
   5 & 10 & $4.73 \times 10^{-2}$ & 10.05 &0\\
   3 & 10 & $14.28$ & 10.12 &0\\
   3 & 15 &$4.65 \times 10^{-2}$& 10.80 &0\\
\end{tabular}
\caption{Results for the experimental setting described in Section \ref{sec_numerical_ellipses}. Observe that the smaller the value of $d$ -- the closer the obstacles are between them -- the environment becomes harder to navigate, i.e. $k$ must be increased to converge to the minimum of $f_0$. }
\label{table_different_distance}
        \end{subtable}
        \hspace{50pt}
        \centering
        \begin{subtable}[t]{0.4\linewidth}
          \begin{center}
\begin{tabular}{ c | c| c |c }
$d$ & $k$ & $\mu_r$& $\sigma_r^2$ \\
\hline
10& 2 &  1.07& $6.53 \times 10^{-3}$ \\
10&15&  1.01& $6.95 \times 10^{-5}$ \\
9& 5&  1.03& $2.10 \times 10^{-3}$ \\
9&15&  1.01& $7.74\times 10^{-4}$ \\
6& 7&  1.19& $1.01 \times 10^{-2}$ \\
6&15&  1.03& $1.59\times 10^{-3}$ \\
5& 10&  1.06& $6.14 \times 10^{-3}$ \\
5&15&  1.05& $2.57\times 10^{-3}$ \\
3&15&  1.06& $3.60\times 10^{-3}$ \\
\end{tabular}
\end{center}
\caption{Mean and variance of the ratio between the path length and the initial distance to the minimum. For each scenario $100$ simulations were considered. Observe that the smaller the value of $d$ the larger the ratio becomes.}
\label{table_statistics}
\end{subtable}
 \end{table*}
In this section we consider $m$ elliptical obstacles in $\mathbb{R}^n$, where $\beta_i(x)$ is of the form \eqref{eqn_ellipses_boundary}, with $n=2$ and $n=3$.  We set the number of obstacle to be $m=2^n$ and we define the external boundary to be a spherical shell of center $x_0$ and radius $r_0$. The center of each ellipsoid is placed the position $d \left(\pm 1, \pm 1, \ldots, \pm 1\right)$ and then we perturb this position by adding a vector drawn uniformly from $[-\Delta, \Delta]^n$, where $0<\Delta<d$. 
The maximum axis of the ellipse --$r_i$ -- is drawn uniformly from  $[r_0/10, r_0/5]$. We build orthogonal matrices $A_i$ for $i=1\ldots m$ where their eigenvalues are drawn from the uniform distribution over $[1,2]$. We verify that the obstacles selected through the previous process do not intersect and if they do, we re draw all previous parameters. For the objective function we consider a quadratic cost given by
\begin{equation}\label{en_quadratic_cost}
f_0(x) = \left(x - x^*\right)^T Q \left(x - x^*\right),
\end{equation}  
where  $Q \in \mathcal{M}^{n\times n}$ is a positive symmetric $n\times n$ matrix. $x^*$ is drawn uniformly over $[-r_0/2, r_0/2] ^n $ and we verify that it is in the free space. Then, for each obstacle we compute the maximum condition number, i.e, the ratio of the maximum and minimum eigenvalues, of $Q$ such that \eqref{eqn_general_condition} is satisfied. Let $N_{cond}$ be the largest condition number that satisfies all the constraints. Then, the eigenvalues of $Q$ are selected randomly from $[1 , N_{cond}-1]$, hence ensuring that $\eqref{eqn_general_condition}$ is satisfied. Finally the initial position is also selected randomly over $[-r_0, r_0] ^n$ and it is checked that it lies on the free space. 

For this experiments we set $r_0=20$ and $\Delta =1$. We run $100$ simulations varying the parameter $d$ -- controlling the proximity of the obstacles-- and $k$. With this information we build Table \ref{table_different_distance}, where we report the number of collisions, the maximal distance of the last iterate to the minimum of $f_0$ and the minimal initial distance to the minimum of $f_0$. As we can conclude from Table \ref{table_different_distance}, the artificial potential \eqref{eqn_partial_navigation_function} provides collision free paths. Notice that the smaller the distance between the obstacles the harder is to navigate the environment and $k$ needs to be further increased to achieve the goal. For instance we observe that setting $k=5$ is sufficient to navigate the world when $d=9$, yet it is not enough to navigate an environment where $d =6$. 
The trajectories arising from artificial potentials typically produce paths whose length is larger than the distance between the initial position and the minimum. We perform a statistical study reporting in Table \ref{table_statistics} the mean and the variance of the ratio between these two quantities. We only consider those values of $d$ and $k$ that always achieve convergence (c.f Table \ref{table_different_distance}). Observe that when the distance $d$ is reduced while keeping $k$ constant the ratio increases. On the contrary if $d$ is maintained constant and $k$ is increased the ratio becomes smaller, meaning that the trajectory approaches the optimal one.  
 In Figure \ref{fig_3d} we simulate one instance of an elliptical world in $\mathbb{R}^3$, with $d=10$ and $k=25$. For four initial conditions we observe that the trajectories succeed to achieve the minimum of $f_0$. 
%

%
%
%
%
%
%
\begin{figure}
\centering
\includegraphics[width=0.5\textwidth]{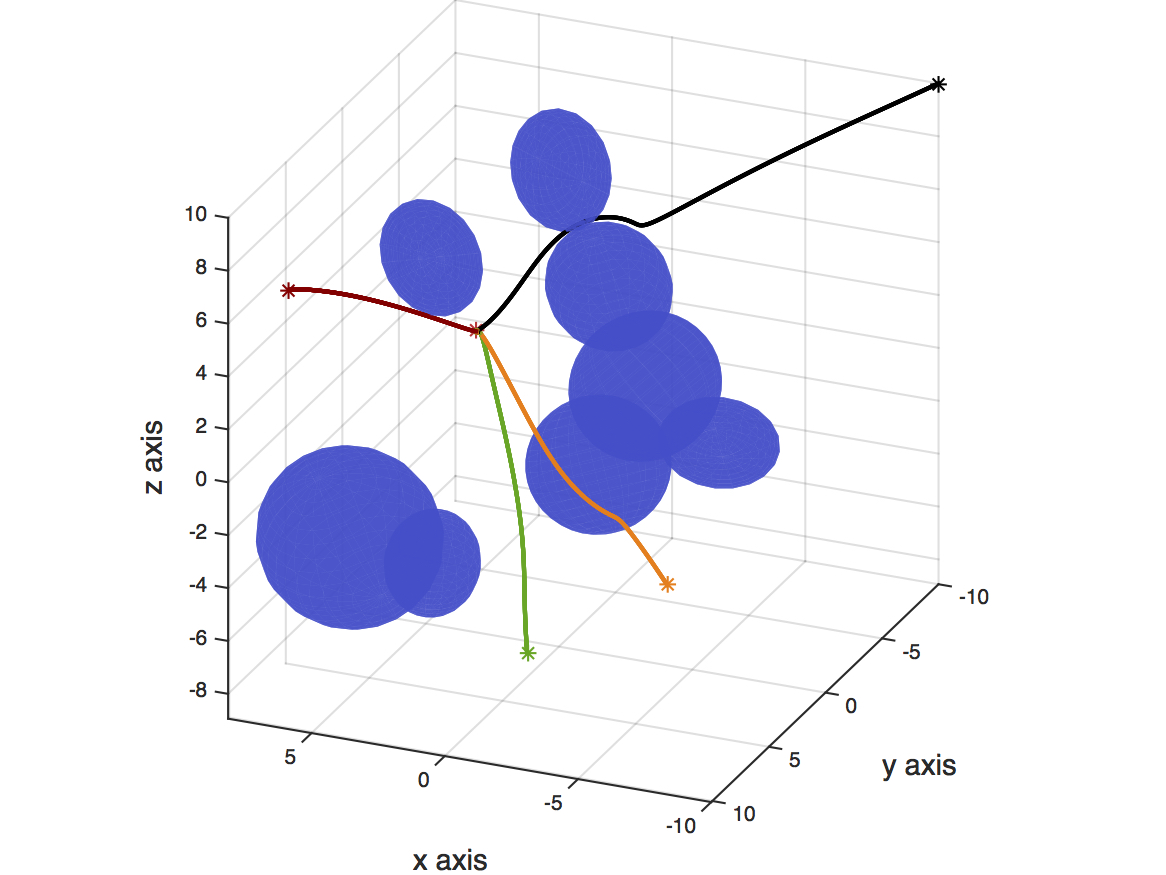}
\caption{Trajectories for different initial conditions in an elliptical world in $\mathbb{R}^3$. As per Theorem \ref{theo_ellipses} and \ref{theo_switched} the trajectory converges to the minimum of the objective function while avoiding the obstacles. In this example we have $d=10$ and $k=25$. }
\label{fig_3d}
\end{figure}
%
\subsection{Egg shaped obstacles}\label{sec_general_convex_obstacle}
In this section we consider the class of egg shaped obstacles. We draw the center of the each obstacle, $x_i$, from a uniform distribution over $[-d/2, d/2] \times [-d/2, d/2]$. The distance between the ''tip'' and the ''bottom'' of the egg, $r_i$, is drawn uniformly over $[r_0/10; r_0/5]$ and with probability $0.5$, $\beta_i$ is
\begin{equation}
\beta_i(x) = \|x-x_i \|^4 - 2r_i \left(  x^{(1)} - x_i^{(1)}\right)^3,
\end{equation}
resulting in a horizontal egg. The superscript $(1)$ refers to first component of a vector. With probability $0.5$ the egg is vertical
\begin{equation}
\beta_i(x) = \|x-x_i \|^4 - 2r_i \left(  x^{(2)} - x_i^{(2)}\right)^3.
\end{equation}
Notice that the functions $\beta_i$ as defined above are not convex on $\mathbb{R}^2$, however their Hessians are positive definite outside the obstacles. To be formal we should define a convex extension of the function inside the obstacles in order to say that the function describing the obstacle is convex. This extension is not needed in practice because our interest is limited to the exterior of the obstacle. In Figure \ref{fig_egg} we observe the level sets of the navigation function and a trajectory arising from \eqref{eqn_approx_flow} when we set $k=25$, $r_0 = 20$ and $d=10$. In this example the hypotheses of Theorem \ref{theo_general} are satisfied, hence the function $\varphi_k$ is a navigation function and trajectories arising from the gradient flow \eqref{eqn_partial_gradient_flow} converge to the optimum of $f_0$ without running into the free space boundary (c.f. Theorem \eqref{theo_switched}). 

\begin{figure}
\centering
\includegraphics[width=0.5\textwidth]{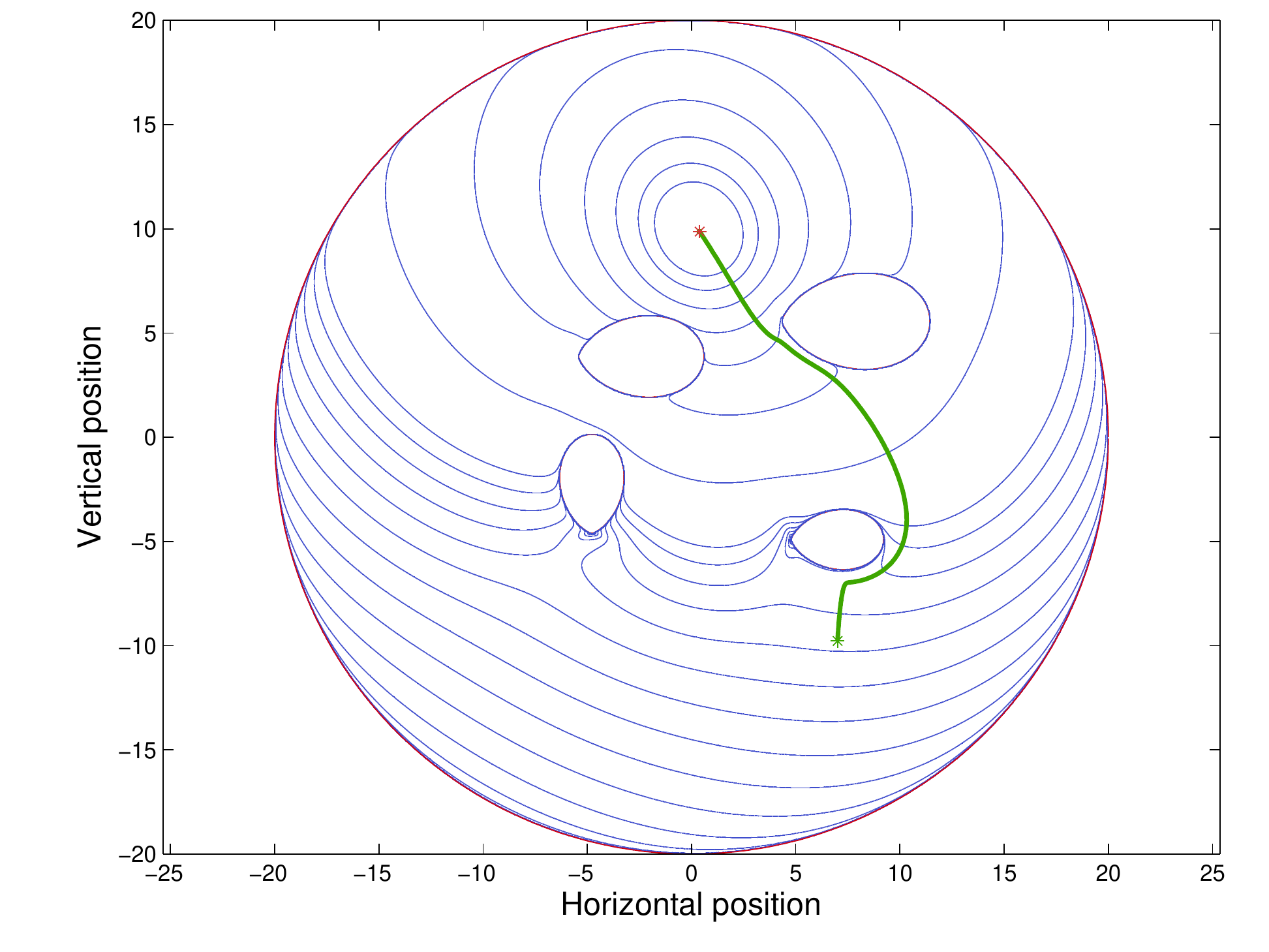}
\caption{Navigation function in an Egg shaped world. As predicted by Theorem \ref{theo_switched} the trajectory arising from \eqref{eqn_approx_flow} converges to the minimum of the objective function $f_0$ while avoiding the obstacles. }
\label{fig_egg}
\end{figure}
%
\subsection{Violation of condition \eqref{eqn_condition_ellipses} }\label{sec_condition_violated}
In this section we generate objective functions such that condition \eqref{eqn_condition_ellipses} is violated. To do so, we generate the obstacles as in Section \ref{sec_numerical_ellipses} and the objective function is such that all the eigenvalues of the Hessian are set to be one, except for the maximum which is set to be $\max_{i=1\ldots m}{N_{cond}} +1$, hence assuring that condition \eqref{eqn_condition_ellipses} is violated for all the obstacles. In this simulation Theorem \ref{theo_ellipses} does not ensures that $\varphi_k$ is a navigation function so it is expected that the trajectory fails to converge. We run $100$ simulations for different values of $d$ and $k$ and we report the percentage of successful simulations in Table \ref{table_violation_condition}. For each value of $d$ the selection of $k$ was done based on Table \ref{table_different_distance}, where $k$ is such that all the simulations attain the minimum of the objective function. Observe that when the distance between the obstacles is decreased the probability of converging to a local minimum different than $x^*$ increases. 
%
%
\begin{table}
  \begin{center}
\begin{tabular}{ c | c| c|c|c|c }
  d & 10 & 9 & 6&5&3\\
  \hline
  k & 2 & 5 & 7&10&15\\
  \hline
Success & $99\%$ &$99\%$&$99\%$&$99\%$&$99\%$ \\
\end{tabular}
\end{center}
\caption{Percentage of successful simulations when the condition guaranteeing that $\varphi_k$ is a navigation function is violated. We observe that as the distance between obstacles becomes smaller the failure percentage increases. }
\label{table_violation_condition}
\end{table}
%
%
\subsection{Double integrator dynamics}\label{sec_dynamics}
In this section we consider a double integrator $\ddot{x}=\tau$ as an example of the dynamics \eqref{eqn_robot_model}
%
and the following control law 
\begin{equation}\label{eqn_control_law}
\tau = -\nabla \varphi_k(x) -K\dot{x} .
\end{equation}
In Figure \ref{fig_dynamics} we observe the behavior of the double integrator when the control law \eqref{eqn_control_law} is used (green trajectories) against the behavior of the gradient flow system \eqref{eqn_partial_gradient_flow} (orange trajectory). Thee light and dark green lines correspond to systems where the damping constant are $K=4\times 10^3$ and $K=5\times 10^3$ respectively. As we can observe the larger this constant the closer the trajectory is to the system without dynamics.
\begin{figure}
\centering
\includegraphics[width=0.5\textwidth]{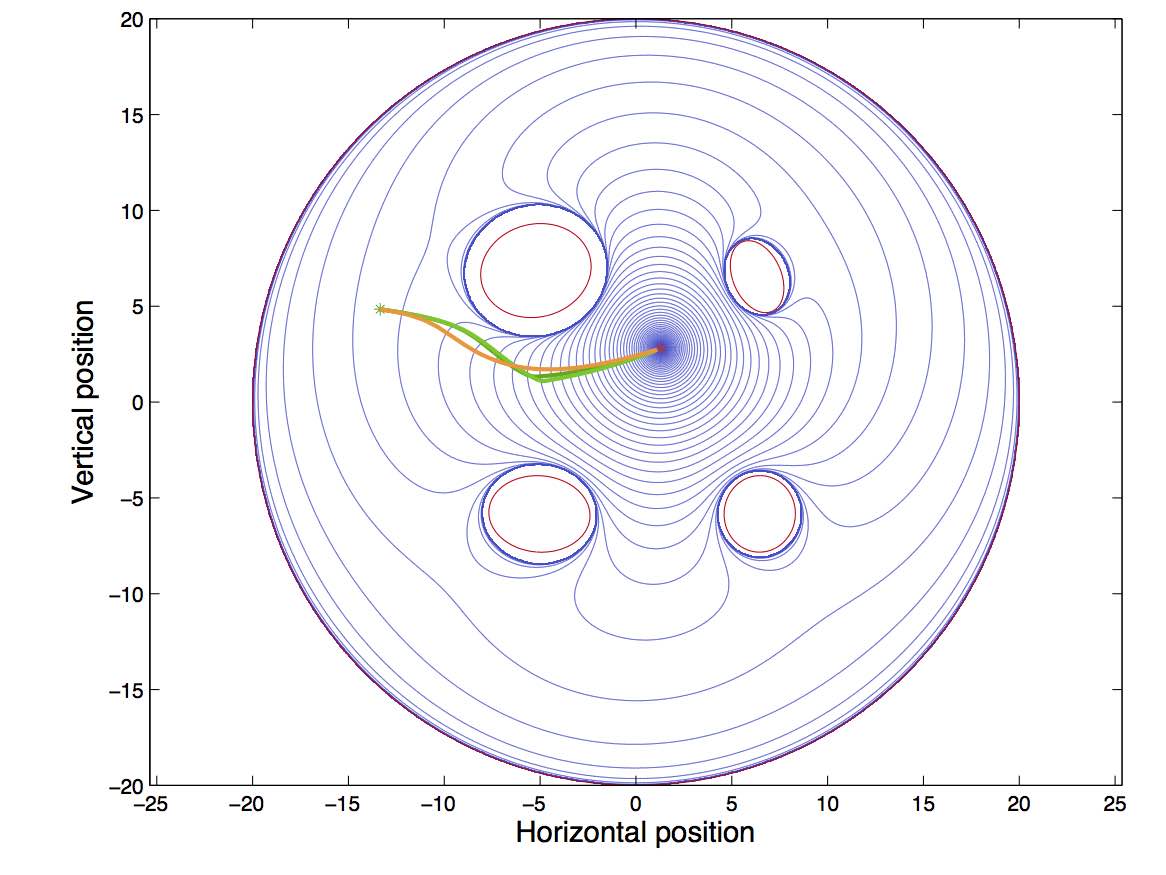}
\caption{In orange we observe the trajectory arising from the system without dynamics (c.f. \eqref{eqn_partial_gradient_flow}). In green we observe trajectories of the double integrator dynamics when we the control law \eqref{eqn_control_law} is applied. The trajectory in dark green has a larger damping constant than the trajectory in light green and therefore it is closer to the trajectory of the system without dynamics.}
\label{fig_dynamics}
\end{figure}
%
\subsection{Differential drive robot}\label{sec_wheeled}
In this section we consider a disk shaped differential drive robot $(x,\theta)\in\mathbb{R}^2\times (-\pi,\pi]$, centered at $x\in\mathbb{R}^2$ with body radius $r>0$ and orientation $\theta\in (-\pi,\pi]$. Its kinematics are  
\begin{equation}\label{eqn_kynematics_wheeled_robot}
\dot{x} = v\left[\begin{array}{c}\cos\theta\\ \sin\theta\end{array}\right],\quad \dot{\theta} = \omega,
\end{equation}
where $v$ and $\omega$ are the linear and angular velocity. The control inputs $\tau_v$ and $\tau_{\omega}$ actuate respectively over their derivatives
\begin{equation}\label{eqn_dynamics_wheeled_robot}
\dot{v} = \tau_v, \quad \dot{\omega} = \tau_\omega.
\end{equation}
Observe that the robot described by \eqref{eqn_kynematics_wheeled_robot} and \eqref{eqn_dynamics_wheeled_robot} is an under actuated example of the general robot \eqref{eqn_robot_model}. Because of the under actuation it is not possible to follow the exact approach described in Remark \ref{rmk_dynamics}. We follow the approach in \cite{tanner2000nonholonomic} to extend the navigation function to a kinematic model of the differential drive robot. Define the desired angle
\begin{equation}
\theta_d = \arg\left(\frac{\partial \varphi_k(x,y)}{\partial x} +i\frac{\partial \varphi_k(x,y)}{\partial y}\right),
\end{equation}
where $\arg(a+ib)$ is the argument of the complex number $a+ib$. Then the commanded speed is
\begin{equation}\label{eqn_control_law_kinematic_velocity}
\begin{split}
  v_c &= -sgn\left(\frac{\partial \varphi_k(x,y)}{\partial x}\cos\theta+\frac{\partial \varphi_k(x,y)}{\partial y}\sin\theta\right)\\
  &\left\{k_v\left[ \left(\frac{\partial \varphi_k(x,y)}{\partial x}\right)^2+\left(\frac{\partial \varphi_k(x,y)}{\partial x}\right)^2\right]\right\}.
\end{split}
\end{equation}
In the above equation $sgn(x)$ is the sign function defined as  $sgn(x)= 1 $ if $\geq0$ and $sgn(x)=-1$ otherwise. The commanded angular speed is then given by
\begin{equation}\label{eqn_control_law_kinematic_angular}
\omega_c = k_{\omega}\left(\theta_d-\theta\right).
  \end{equation}
We propose the following control law to extend the kinematic by setting the input of the linear and angular accelerations to 
\begin{equation}\label{eqn_control_law}
\tau_v = -v_c - k_{v,d}v \quad\mbox{and}\quad  \tau_\omega = -\omega_c-k_{\omega,d}\omega.
\end{equation}
We emphasize that the proposed control does not provide stability guarantees and we are presenting it as an illustration on how to extend the navigation function to systems with dynamics. In Figure \ref{fig_wheeled} we depict in green the trajectories of the kinematic differential drive robot \eqref{eqn_kynematics_wheeled_robot}, when the control law is given by \eqref{eqn_control_law_kinematic_velocity} and \eqref{eqn_control_law_kinematic_angular}. In orange we depict the trajectories of the dynamic differential drive robot (\eqref{eqn_kynematics_wheeled_robot} and \eqref{eqn_dynamics_wheeled_robot}, when the control law is given by \eqref{eqn_control_law}. In these examples we observe that for $k_v=k_\omega=1$ and $k_{v,d}=4$ and $k_{\omega,d}=10$ the wheeled robot succeeds in reaching the minimum of the objective function while avoiding the obstacles.

\begin{figure}
\centering
\includegraphics[width=0.4\textwidth]{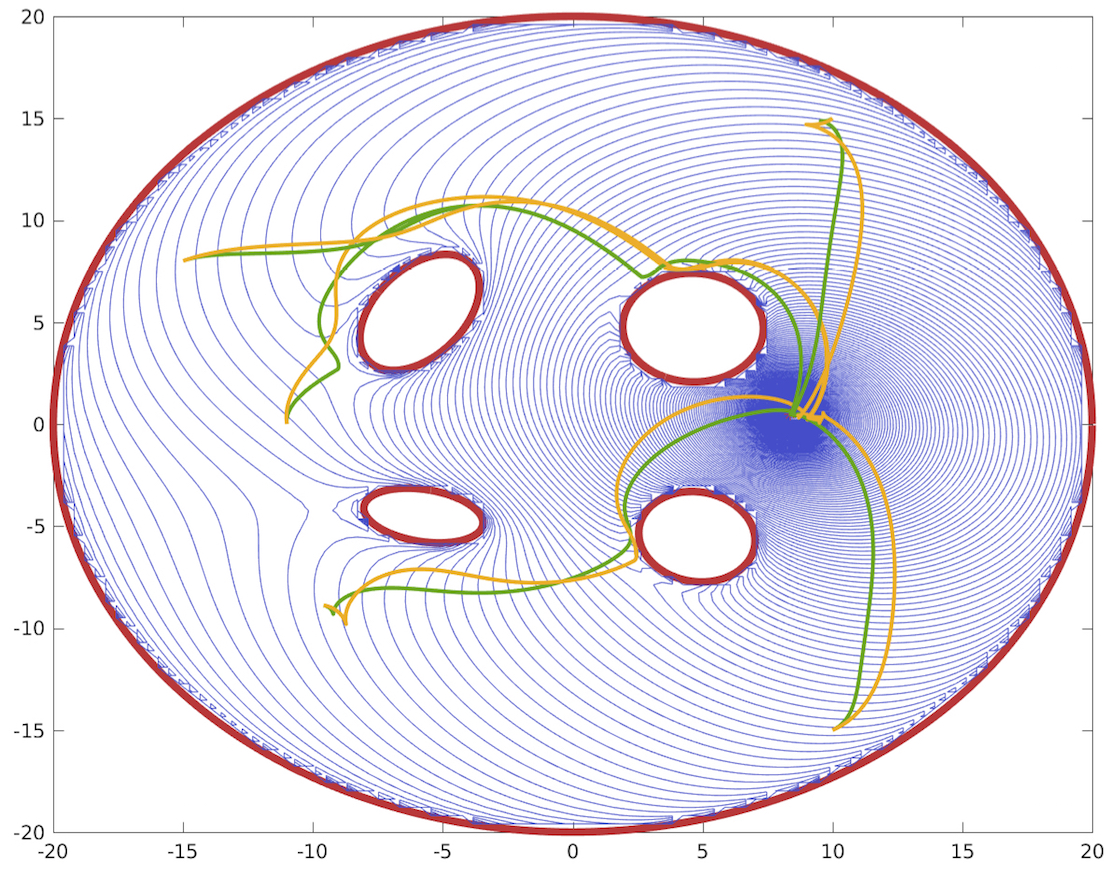}
\caption{In green we depict the trajectories of the kinematic differential drive robot \eqref{eqn_kynematics_wheeled_robot} , when the control law is given by \eqref{eqn_control_law_kinematic_velocity} and \eqref{eqn_control_law_kinematic_angular}. In orange we depict the trajectories of the dynamic differential drive robot ( \eqref{eqn_kynematics_wheeled_robot} and \eqref{eqn_dynamics_wheeled_robot} , when the control law is given by \eqref{eqn_control_law}. In both cases we select $k_v=k_\omega=1$ and for the dynamic system $k_{v,d}=4$ and $k_{\omega,d}=10$ . As it can be observed the agent reaches the desired configuration while avoiding the obstacles.} 
\label{fig_wheeled}
\end{figure}

%
\section{Conclusions}
We considered a set with convex holes in which an agent must navigate to the minimum of a convex potential. This function is unknown and only local information about it was used, in particular its gradient and its value at the current location. We defined an artificial potential function and we showed that under some conditions of the free space geometry and the objective function, this function was a navigation function. Then a controller that moves along the direction of the negative gradient of this function ensures convergence to the minimum of the objective function while avoiding the obstacles. In order to avoid the need of knowledge about the environment a switched controller based on the previous navigation function is defined. This controller only takes into account information about the obstacles that the agent has visited. Numerical experiments support the theoretical results.

\appendix
\section{Appendices}
\subsection{Proof of Lemma \ref{lemma_critical_points_new}}\label{ap_lemma_critical_points_new}
Since $\varphi_k$ is twice continuously differentiable and its maximum is attained in the boundary of the compact set $\ccalF$ (c.f. Lemma \ref{lemma_smooth}) it must be the case that there exists $x_c\in\mbox{int}(\ccalF)$ such that $\nabla \varphi_k(x_c) =0$. In Lemma \ref{lemma_smooth} it was argued that for all $x\in\ccalF$ it holds that  $f_0^k(x)+\beta(x)>0$. Hence $\nabla \varphi_k(x_c)=0$ (c.f. \eqref{eqn_nabla_phi}) if and only if 
\begin{equation}\label{eqn_critical_point_condition_apendix}
k\beta(x_c)\nabla f_0(x_c) = f_0(x_c)\nabla \beta(x_c)
  \end{equation}
In cases where $\nabla \beta(x^*) = 0$ or $f_0(x^*)=0$ then the previous equation is satisfied for $x_c = x^*$ and $x^*$ is a critical point. By virtue of Lemma \ref{lemma_critical_points_interior} there are not critical points in the boundary of the free space, hence the left hand size of the above equation is not zero for any $x_c \neq x^*$.
Since $x^*\in\mbox{int}(\ccalF)$ (see Assumption \ref{assum_objective_function}) there exists $\delta_0>0$ such that for any $\delta \in (0,\delta_0]$ we have 
\begin{equation}
\left\{x\in\ccalF \big|\beta(x)<\delta\right\}\cap \left\{x\in\ccalF \big|\|\nabla f_0(x)\|<\delta\right\} = \emptyset
\end{equation}
Since $f_0$ is non negative and both $f_0$, $\beta$ are twice continuously differentiable (see Assumption \ref{assum_objective_function}) and $\ccalF$ is a compact set, there exists $C>0$ such that $f_0(x)\|\nabla \beta(x)\|<C$ for all $x\in\ccalF$. Hence, from \eqref{eqn_critical_point_condition_apendix} we have that for any $\delta_1\in(0,\delta_0]$ there exists $K_1>0$ such that if $k>K_1$ then
\begin{equation}
\beta(x_c)\|\nabla f_0(x_c)\| < \delta_1^2.
  \end{equation}
By construction both $\beta(x_c)$ and $\|\nabla f_0(x_c)\|$ cannot be smaller than $\delta_1$ and if they are both larger than $\delta_1$ then the above inequality is violated. Hence, either $\beta(x_c)<\delta_1$ or $\|\nabla f_0(x_c)\|<\delta_1$. Moreover, using the same argument for the individual functions $\beta_i(x)$, since the obstacles do not intersect (c.f. Assumption \ref{assum_obstacles}) there exists $\varepsilon_0'>0$ such that for any $\varepsilon \in (0,\varepsilon_0']$ there exists $K_0\prime(\varepsilon)>0$ such that if $k>K_0\prime(\varepsilon)$ then $x_c$ is such that either $\|\nabla f_0(x_c)\|<\varepsilon$ or for exactly one $i$ we have that $\beta_i(x_c)<\varepsilon$. We next show that the critical points cannot be pushed towards the external boundary of the free space. Assume that for all $\varepsilon \in (0,\varepsilon_0']$ there exists $K_0\prime(\varepsilon)$ such that for all $k>K_0\prime(\varepsilon)$ there is a critical point $x_c$ satisfying $\beta_0(x_c)<\varepsilon$. Let us write the gradient of $\nabla \beta(x_c)$ as in \eqref{eqn_nabla_obstacle}    
\begin{equation}\label{eqn_lemma4aux}
\nabla \beta(x_c) = \bar{\beta}_0(x_c) \nabla \beta_0(x_c) + \beta_0(x_c) \nabla \bar{\beta}_0(x_c)
\end{equation}
Since the workspace is a convex set is a super level set of a concave function (c.f.\eqref{eqn_beta0} ) it holds that $\nabla \beta_0(x_s)^T (x_s-x^*)<0$. Since $\nabla \bar{\beta_0}$ is continuous (c.f. Assumption \ref{assum_obstacles}), over the compact set $\ccalF$ it is bounded. Then, choose $\varepsilon_0<\varepsilon_0\prime$ such that $\nabla \beta(x_s)^T (x_s-x^*)<0$. It follows from \eqref{eqn_critical_point_condition_apendix} that at a critical point $\nabla \beta(x_s)$ and $\nabla f_0(x_s)$ point in the same direction and therefore there exists $K_0(\varepsilon_0)>0$ such that if $k>K_0(\varepsilon_0)$ then $\nabla f_0(x_s)^T(x_s-x^*)<0$. The latter however contradicts the first order condition of convexity (see e.g. \cite{boyd2004convex}). Hence, for any $\varepsilon< \varepsilon_0$ there exists $K_0(\varepsilon)>0$ such that if $k>K_0(\varepsilon)$ for any critical point we have that $\beta_0(x_c)>\varepsilon_0$.
%
\subsection{Proof of Lemma \ref{lemma_saddle_points}}\label{ap_proof_saddles}
Let $x_s$ be a critical point such that $\beta_i(x_s)<\varepsilon_0$ for some $i=1\ldots m$ where $\varepsilon_0$ is that of Lemma \ref{lemma_critical_points_new} and let $v$ be a unit vector normal to $\nabla \beta(x_s)$. If we prove that $v^T\nabla^2\varphi_k(x_s) v <0$ then $x_s$ is not a local minimum. Differentiating \eqref{eqn_nabla_phi} and using the fact that for a critical point \eqref{eqn_critical_point_condition_apendix} holds, we can write 
\begin{equation}\label{eqn_hess_phi}
\begin{split}
&\nabla^2 \varphi_k(x_s) = \left(f_0^k(x_s)+\beta(x_s)\right)^{-1 -\frac{1}{k}}  \left(\beta(x_s)\nabla^2 f_0(x_s)+\phantom{(1-\frac{1}{k})\nabla f_0(x_s)\nabla \beta(x_s)^T} \right. \\
&\left. (1-\frac{1}{k})\nabla f_0(x_s)\nabla \beta(x_s)^T -\frac{f_0(x_s)\nabla^2\beta(x_s)}{k} \right).
\end{split}
\end{equation}
In Lemma \ref{lemma_smooth} we argued that for all $x\in\ccalF$ it holds that $f_0^k(x)+\beta(x)>0$. In addition, since by definition $v^T\nabla \beta(x_s) =0 $,  we have that $v^T\nabla^2\varphi_k(x_s)v<0$ if and only if 
\begin{equation}\label{eqn_sign_eval}
k\beta(x_s)v^T\nabla^2 f_0(x_s)v  - f_0(x_s)v^T\nabla^2\beta(x_s)v<0.
\end{equation}
Since $x^*:=\argmin f_0(x)$, then $\nabla f_0(x^*)=0$ and we can use  \eqref{eqn_strong_convexity2} to lower bound $\nabla f_0^T(x_s)(x_s-x^*)$ as 
\begin{equation}\label{eqn_inter_cond_kbeta1}
\lambda_{\min}\|x_s-x^* \|^2 \leq \nabla f_0^T(x_s)(x_s-x^*).
\end{equation}
Since $x_s$ is a critical point \eqref{eqn_critical_point_condition_apendix} holds. Multiply both sides of the equation by $(x_s - x^*)$ to write
\begin{equation}\label{eqn_inter_cond_kbeta}
k \beta(x_s)\nabla f_0^T(x_s)(x_s-x^*) = f_0(x_s)\nabla\beta(x_s)^T(x_s-x^*).
\end{equation}
From Lemma \ref{lemma_critical_points_new} we have that $\|\nabla f_0(x_s)\|>\varepsilon_0$ which is independent of $k$, hence $\|x_s-x^*\|$ is bounded away from zero by a constant independent of $k$. Therefore we can upper bound $k\beta(x_s)$ by
\begin{equation}\label{eqn_inter_cond_kbeta1}
k \beta(x_s) \leq f_0(x_s)\frac{ \nabla\beta(x_s)^T(x_s-x^*)}{\lambda_{\min}\|x_s - x^* \|^2}.
\end{equation}
Substituing $\nabla\beta(x_s)$ in \eqref{eqn_inter_cond_kbeta1} by its expression in \eqref{eqn_nabla_obstacle} yields
\begin{equation}\label{eqn_inter_cond_kbeta2}
\begin{split}
k \beta(x_s) & \leq \frac{f_0(x_s)}{\lambda_{\min}\|x_s - x^* \|^2} \bar{\beta}_i(x_s)    \nabla \beta_i(x_s)^T (x_s-x^*)  \\
& +  \frac{f_0(x_s)}{\lambda_{\min}\|x_s - x^* \|^2} \beta_i(x_s)    \nabla \bar{\beta}_i(x_s)^T (x_s-x^*).
\end{split}
\end{equation}
We argue next that the second term of \eqref{eqn_inter_cond_kbeta2} is bounded by a constant. As argued in the previous paragraph $\|x_s-x^*\|$ is bounded away from zero by a constant independent of $k$. In addition the remaining factors are the product of continuous functions in a bounded set, thus they are uniformly bounded as well. Let $B>0$ be a constant bounding the terms multiplying $\beta_i(x_s)$ in the second term of \eqref{eqn_inter_cond_kbeta2}, i.e, 
\begin{equation}\label{eqn_aux_bound}
\begin{split}
\frac{f_0(x_s)}{\lambda_{\min}\|x_s - x^* \|^2}  \nabla \bar{\beta}_i(x_s)^T (x_s-x^*) \leq B.
\end{split}
\end{equation}
Now, let us focus on the second term of \eqref{eqn_sign_eval}, in particular the Hessian of $\beta(x_s)$ can be computed by differentiating  \eqref{eqn_nabla_obstacle}
\begin{align}\label{eqn_beta_hessian}
\nabla^2 \beta(x_s) &= \beta_i(x_s) \nabla^2 \bar{\beta}_i (x_s) +\bar{\beta}_i(x_s) \nabla^2 \beta_i(x_s) \nonumber \\
& +2\nabla \beta_i(x_s) \nabla ^T \bar{\beta}_i(x_s).
\end{align}
It follows from the result of Lemma \ref{lemma_critical_points_new} and the non negativity of the objective function (c.f. Assumption \ref{assum_objective_function} that both $f_0(x_s)$ and $\bar{\beta}_i(x_s)$ are bounded away form zero. Then, combine  \eqref{eqn_nabla_obstacle} and \eqref{eqn_critical_point_condition_apendix} to express the gradient of $\nabla \beta_i(x_s)$ as
\begin{equation}\label{eqn_grad_betai}
\nabla \beta_i(x_s) = k\beta_i(x_s) \frac{\nabla f_0(x_s)}{f_0(x_s)} - \beta_i(x_s) \frac{\nabla \bar{\beta}_i(x_s)}{\bar{\beta}_i(x_s)}.
\end{equation}
Recall from \eqref{eqn_critical_point_condition_apendix} that at the critical point $\nabla \beta(x_s)$ and $\nabla f_0(x_s)$ are point in the same direction, thus $v^T\nabla f_0(x_s) = 0$ since $v$ is perpendicular to $\nabla \beta(x_s)$. Hence 
\begin{equation}\label{eqn_key_lemma_aux}
v^T \nabla \beta_i(x_s) =  -\beta_i(x_s) v^T\frac{\nabla \bar{\beta}_i(x_s)}{\bar{\beta}_i(x_s)}.
\end{equation}
Combine \eqref{eqn_beta_hessian} and \eqref{eqn_key_lemma_aux} to evaluate the quadratic form associated with the Hessian of $\beta(x_s)$ along the direction $v$
\begin{equation}
\begin{split}
v^T \nabla^2 \beta(x_s) v &= v^T\nabla^2\beta_i(x_s) v\bar{\beta}_i(x_s) \\
&+ \beta_i(x_s) \left(v^T\nabla^2 \bar{\beta}_i(x_s)v - 2\frac{\| v^T \nabla \bar{\beta}_i(x_s)\|^2}{\bar{\beta}_i(x_s)} \right).
\end{split}
\end{equation} 
In the above equation the absolute value of the function multiplying $\beta_i(x_s)$ is upper bounded by a constant independent of $k$. Let $B'>0$ be this constant. Then, the second term of \eqref{eqn_sign_eval} is upper bounded by
\begin{equation}\label{eqn_second_term_bound}
\begin{split}
-f_0(x_s)v^T \nabla^2 \beta(x_s) v  \leq \\
  - v^T \nabla^2\beta_i(x_s) v\bar{\beta}_i(x_s)f_0(x_s) + \beta_i(x_s) B'.
\end{split}
\end{equation} 
Use the bounds \eqref{eqn_inter_cond_kbeta2}, \eqref{eqn_aux_bound} and \eqref{eqn_second_term_bound} and the fact tht $v^T\nabla f_0(x_s)v\leq \lambda_{\max}$ to bound the left hand side of \eqref{eqn_sign_eval} by
\begin{equation}\label{eqn_almost_final_bound_lemma}
\begin{split}
 k\beta(x_s)&v^T\nabla^2 f_0(x_s)v  - f_0(x_s)v^T\nabla^2\beta(x_s)v  \\
\leq & v^T\nabla^2 f_0(x_s)v \frac{f_0(x_s) \bar{\beta}_i(x_s)}{\lambda_{\min}\|x_s - x^* \|^2} \nabla \beta_i(x_s)^T (x_s-x^*)  \\
-&v^T \nabla^2 \beta_i(x_s) v f_0(x_s)\bar{\beta}_i(x_s) +    \beta_i(x_s) \left(B\lambda_{\max}+B'\right).
\end{split}
\end{equation}
As argued previously $\beta_j(x_s)$ is bounded away from zero by a constant independent of $k$ for all $j\neq i$. The same holds for $f_0(x_s)$.  Then, we have that $v^T \nabla^2 \varphi_k(x_s)v<0$ if 
\begin{equation}
\begin{split}
 v^T\nabla^2 f_0(x_s)v \frac{\nabla \beta_i(x_s)^T (x_s-x^*) }{\lambda_{\min}\|x_s - x^* \|^2}  \\
- v^T\nabla^2 \beta_i(x_s)v \leq -\beta_i(x_s)B'', 
\end{split}
\end{equation}
where $B''>0$ is a bound for $(B\lambda_{\max}+B')/(\bar{\beta}_i(x_s)f_0(x_s))$. From Assumption \ref{assum_objective_function} we have that  $v^T\nabla^2 f_0(x_s)v  \leq \lambda_{\max}$ and $v^T \nabla^2 \beta_i(x_s) v\geq \mu_{\min}^i$, then $v^T\nabla^2 \varphi (x_s) v <0$ if
\begin{equation}
\frac{\lambda_{\max}}{\lambda_{\min}} \frac{\nabla \beta_i(x_s)^T (x_s-x^*) }{\|x_s - x^* \|^2}  -\mu_{\min}^i \leq -\beta_i(x_s)B''.
\end{equation}
By hypothesis the left hand side of the above equation is strictly negative in the boundary of the obstacle, and the right hand side takes the value zero. Therefore there exists $\varepsilon_1>0$ such that for any $\varepsilon \in (0, \varepsilon_1]$ if $\beta_i(x_s)<\varepsilon$ the above inequality is satisfied. Thus, from the result in Lemma \ref{lemma_critical_points_interior} there exists some $K_1(\varepsilon)>K_0(\varepsilon)$ such that for any $k>K_1(\varepsilon)$ the critical point is not a minimum.  

\subsection{Proof of Lemma \ref{lemma_polar_general}}\label{ap_general_proof_minimum}
Since $\varphi_k(x)$ is a twice continuously differentiable function and it attains its maximum at the boundary of a compact set (see Lemma \ref{lemma_smooth}) it must have a minimum in the interior of $\ccalF$. In virtue of Lemma \ref{lemma_saddle_points} for any $\varepsilon<\varepsilon_1$ there exists $K_1(\varepsilon)>0$ such that if $k>K_1(\varepsilon)$ the critical points $x_c$ such that $\beta_i(x_c)<\varepsilon$ are not local minima. Hence the minimum for $\varphi_k(x)$ is such that $\|\nabla f_0(x_c)\| < \varepsilon$. We next show that any critical point satisfying $\|\nabla f_0(x_c)\| < \varepsilon$ is a non degenerate minimum. Using the same arguments as in Lemma \ref{lemma_saddle_points} we have that $\nabla^2 \varphi_k(x_c)>0$ if and only if 
\begin{equation}\label{eqn_minimum_intermidiate}
\begin{split}
\beta(x_c)\nabla^2 f_0(x_c) &+ (1-\frac{1}{k})\nabla \beta(x_c) \nabla f_0^T(x_c) \\
&-\frac{f_0(x_c)}{k}\nabla^2\beta(x_c)>0.
\end{split}
\end{equation}
Since  $\|\nabla f_0(x_c)\|<\varepsilon<\varepsilon_0$ it follows from Lemma \ref{lemma_critical_points_new} that each $\beta_i(x_c)>\varepsilon_0$ and therefore $\beta(x_c)>\varepsilon_0^{m+1}$. Hence the first term in the previous equation satisfies
\begin{equation}
\beta(x_c)\nabla^2f_0(x_c)\geq \lambda_{\min}\varepsilon_0^{m+1}I>0.
  \end{equation}
From \eqref{eqn_critical_point_condition_apendix} it follows that  $\nabla f_0(x_c)$ and $\nabla \beta(x_c)$ point in the same direction, thus the second term in \eqref{eqn_minimum_intermidiate} is a positive semi definite matrix for any $k>1$. Therefore for the $\nabla^2 \varphi_k(x_c)$ to be positive definite it suffices that  
\begin{equation}
\frac{f_0(x_c)}{k}\nabla^2\beta(x_c)<\lambda_{\min}\varepsilon_0^{m+1}I.
  \end{equation}
Since $f_0$ and $\beta$ are twice continuously differentiable (see Assumption \ref{assum_objective_function}) $f_0(x_c)\nabla^2 \beta(x_c)$ is bounded by a constant independent of $k$ because the free space is compact. Therefore there exists $K_2\prime(\varepsilon_0)>1$ such that if $k>K_2\prime(\varepsilon_0)$, the above equation holds and therefore any critical point satisfying $\|\nabla f_0(x_c)\|<\varepsilon$ is a minimum. We are left to show that the minimum is unique. Let $c$ be such that for any $x\in \ccalF$ such that $f_0(x)=c$ then $\|\nabla f_0(x_c)\|<\varepsilon_0$ and define the set $\Omega_c = \left\{x\in\ccalF \big| f_0(x)=c>f_0(x^*) \forall i=0\ldots m\right\}$. By definition of the previous set and because the previous discussion all critical points in $\Omega_c$ are minima. We show next that for large enough $k$, $\Omega_c$ is positively invariant for the flow $\dot{x}=-\nabla \varphi_k(x)$. Compute the derivative of $f_0(x)$ along the trajectories of the flow and evaluate on the boundary of $\Omega_c$
\begin{equation}
\dot{f}_0(x) = -\nabla f_0(x)^T\nabla \varphi_k(x).
  \end{equation}
The previous inner product is negative if and only if
\begin{equation}
\beta(x) \|\nabla f_0(x)\|^2 - \nabla f_0(x)^T\nabla \beta(x) \frac{f_0(x)}{k}>0.
\end{equation}
Observe that first term in the above equation is lower bounded by a constant independent of $k$ in $\partial\Omega_c$ since $c>f_0(x^*)$ and $\beta_i(x)>\varepsilon_0$. Moreover since $\beta$ and $f_0$ are twice continuously differentiable the second term in the previous equation is lower bounded by $-C/k$, where $C$ is independent of $k$. Therefore there exists $K_2\prime\prime(\varepsilon_0)>1$ such that if $k>K_2\prime\prime(\varepsilon_0)$, then $\Omega_c$ is positively invariant, hence the limit set of the flow $\dot{x} = -\nabla \varphi_k(x)$ restricted to $ \Omega_c$ converges to a local minimum. If there were more than one degenerate minimum in $\Omega_c$, since the stable manifold of minimums are open sets, then it would be possible to write $\partial \Omega_c$ as a disjoint union of open sets -- in the topology relative to the boundary of $\Omega_c$. This contradicts the connexity of the boundary. Hence, for any $\varepsilon>0$ there exists $K_2(\varepsilon)=\max\left\{K_1(\varepsilon),K_2\prime(\varepsilon),K_2\prime\prime(\varepsilon) \right\}$ such that if $k>K_2(\varepsilon)$ then $\varphi_k$ is polar with minimum at $\bar{x}$, where $\| \bar{x}-x^*\|<\varepsilon$. Finally from the discussion in Lemma \ref{lemma_critical_points_interior} we have that $\bar{x}=x^*$ if $f_0(x^*)=0$ or $\nabla \beta(x^*)=0$.
%
%
%
%
%
%
%
%
%
\subsection{Proof of Theorem \ref{theo_ellipses}}\label{ap_upper_bound_ellipses_proof}
In the particular case where the functions $\beta_i$ take the form \eqref{eqn_beta_i}, the condition \eqref{eqn_general_condition} of the general Theorem \ref{theo_general} yields
\begin{equation}\label{eqn_first_cond_ellipses}
\frac{\lambda_{\max}}{\lambda_{\min}}\frac{(x_s-x_i)^TA_i(x_s-x^*)}{\| x_s-x^*\|^2} - \mu_{\min}^i <0.
\end{equation}
Since $A_i$ is positive definite, there exists $A_i^{1/2}$ such that
\begin{equation}
A_i = \left(A_i^{1/2}\right)^T A_i ^{1/2}. 
\end{equation}
Consider the change of variables $z = A_i^{1/2}x$, and write 
\begin{equation}
\frac{(x_s-x_i)^TA_i(x_s-x^*)}{\| x_s-x^*\|^2}=\frac{(z_s-z_i)^T(z_s-z^*)}{\| A_i^{-1/2}\left(z_s-z^*\right)\|^2}. 
\end{equation}
Denote by $\mu_{\max}^i$ the maximum eigenvalue of the matrix $A_i$
\begin{equation}
\frac{1}{\mu_{\max}^i}\|\left(z_s-z^*\right)\|^2\leq \| A_i^{-1/2}\left(z_s-z^*\right)\|^2.   
\end{equation}
Use the above inequality to bound the left hand side of \eqref{eqn_first_cond_ellipses}
\begin{equation}\label{eqn_intermediate_lemma3}
\begin{split}
\frac{\lambda_{\max}}{\lambda_{\min}}&\frac{(x_s-x_i)^TA_i(x_s-x^*)}{\| x_s-x^*\|^2} - \mu_{\min}^i\\
&\leq \frac{\lambda_{\max}}{\lambda_{\min}} \frac{(z_s-z_i)^T(z_s-z^*)}{\| z_s-z^*\|^2} \mu_{\max}^i-\mu_{\min}^i.
\end{split}
\end{equation}
The change of coordinates transforms the elliptical obstacle in a sphere of radius $r_i ( \mu_{\min}^i )^{1/2}$ since the function $\beta_i$ takes the following form for the variable $z$ 
\begin{equation}
\beta_i(z) = \| z - z_i \| ^2 - r_i^2 \mu_{\min}^i.
\end{equation}
Since the obstacle is after considering the change of coordinate a circle we define for convenience the radial direction $\hat{e}_r$, whit $\|\hat{e}_r\|=1$. Let $\theta$ be the angle between $\hat{e}_r$ and the direction $z_i - z^*$. Further define $\tilde{r}$ to be the distance between the critical point $z_s$ and $z_i$. Notice that if $|\theta|\leq\pi/2$ then 
\begin{equation}
\frac{(x_s-x_i)^T (x_s-x^*)}{\|x_s - x^* \|^2} \leq 0,
\end{equation}
and in that case the right hand side of \eqref{eqn_intermediate_lemma3} is negative which completes the proof of the lemma. However if  $|\theta|>\pi/2$ then the term under consideration is positive. In particular the larger the norm of $\tilde{r}$ the larger the value. Hence define $\tilde{r}_{\max} = r_i(\mu_{\min}^i)^{1/2} + \varepsilon$, and the following bound holds
\begin{equation}\label{eqn_geometry}
\frac{(z_s-z_i)^T (z_s-z^*)}{\|z_s - z^* \|^2} \leq  \frac{\tilde{r}_{\max}(\tilde{r}_{\max}-d_i\cos\theta)}{\tilde{d_i}^2+\tilde{r}_{\max}^2 -2\tilde{d_i} \tilde{r}_{\max}\cos\theta},
\end{equation}
where $\tilde{d}_i$ is the distance between $z_s$ and $z^*$. 
Differentiating the right hand side of the above equation with respect to $\theta$ we conclude that its critical points are multiples of $\pi$. Notice that for multiples of $\pi$ of the form $2k\pi$, with $k\in \mathbb{Z}$ will correspond to negative values and  and for multiples of $\pi$ of the form $(2k+1)\pi$ with $k \in \mathbb{Z}$, we have that
\begin{equation}
RHS(2k\pi+1)= \frac{\tilde{r}_{\max}(\tilde{r}_{\max}+\tilde{d_i})}{\left(\tilde{d_i}+\tilde{r}_{\max}\right)^2} =\frac{\tilde{r}_{\max}}{\tilde{d}_i+\tilde{r}_{\max}}
\end{equation}
Combine the previous bound with \eqref{eqn_intermediate_lemma3} to upper bound \eqref{eqn_first_cond_ellipses}
\begin{equation}
\begin{split}
\frac{\lambda_{\max}}{\lambda_{\min}}&\frac{(x_s-x_i)^TA_i(x_s-x^*)}{\| x_s-x^*\|^2}\mu_{\max}^i - \mu_{\min}^i \\
&\leq \frac{\lambda_{\max}}{\lambda_{\min}}\frac{\tilde{r}_{\max}}{\tilde{d}_i+\tilde{r}_{\max}}\mu_{\max}^i- \mu_{\min}^i.
\end{split}
\end{equation}
Notice than a lower bound for that distance is given by $\tilde{d}_i \geq \mu_{\min}^id_i$.  Notice that since $z_s$ can be placed arbitrarily close to the boundary of the obstacle $\mathcal{O}_i$ we have that $\tilde{r} \leq r_i(\mu_{\min}^i)^{1/2} + \varepsilon$. To complete the proof observe that 
\begin{equation}
\frac{\tilde{r}_{\max}}{\tilde{d}_i+\tilde{r}_{\max}} = \frac{r_i +\frac{\varepsilon}{ \mu_{\min}^i}}{d_i+r_i +\frac{\varepsilon}{ \mu_{\min}^i}},
\end{equation}
hence since $\varepsilon$ can be made arbitrarily small by increasing $k$ we have tha \eqref{eqn_first_cond_ellipses} holds if   
\begin{equation}
\frac{\lambda_{\max}}{\lambda_{\min}} \frac{\mu_{\max}^i}{ \mu_{\min}^i} < 1 +\frac{d_i}{r_i}.
\end{equation}
 Thus condition \eqref{eqn_general_condition} takes the form stated in the theorem. 
%
\subsection{Proof of Theorem \ref{theo_switched}}\label{ap_theo_switched_proof}
Let us consider the evolution of the dynamical system \eqref{eqn_partial_gradient_flow} from some time $t_0>0$. Notice that if \eqref{eqn_general_condition} holds, then in virtue of Theorem \ref{theo_general} for large enough $k$ the function $\varphi_{k,\ccalA_c(t_0)}(x)$ defined in \eqref{eqn_partial_navigation_function} is a navigation function for the set $\ccalF_{\ccalA_c(t_0)} =  \mathcal{X} \setminus \bigcup_{i}\mathcal{O}_{i\in\ccalA_c(t_0)}$. On one hand, this ensures the avoidance of the obstacles $\ccalO_i$ with $i\in \ccalA_c(t_0)$, furthermore it ensures convergence to $x^*$ -- or to a point arbitrarily close to $x^*$-- unless a new obstacle is visited. If the first happens the proof is completed. In the second case, we need to show that the time lapsed until the agent reaches the neighborhood of a new obstacle is finite. This being the case it would take a finite time $T\geq 0$ to visit all obstacles before having $\varphi_{k,\ccalA_c(t)}(x) = \varphi_k(x)$ for all $x\in \ccalF$. Then for any $t\geq T$ we are in the situation where the obstacles are known and Theorem \ref{theo_general} holds, which completes the proof. Let $t_f$ be the first instant in which the agent reaches the $c$-neighborhood of an obstacle of which he is not aware. Formally, this is
\begin{equation}
t_f = \min \left\{t > t_0\big| \beta_j(x(t)) \leq c \quad \mbox{for some} \quad j \notin \ccalA_c(t_0)(x) \right\}.
\end{equation}
Notice that by the definition of the time $t_f$ we have that $\ccalA_c(t) = \ccalA_c(t_0)$ for all $t\in[t_0,t_f)$. And therefore $\varphi_{k,\ccalA_c(t)}(x) =\varphi_{k,\ccalA_c(t_0)}(x)$ is a navigation function for the free space $\ccalF_{\ccalA_c(t_0)}=\ccalX \setminus\bigcup_{i\in \ccalA_c(t)}$ for all $t\in[t_0,t_f)$. Therefore the critical points of the function \eqref{eqn_partial_navigation_function} are arbitrarily close to $x^*$ or to the obstacles $\ccalO_i$ with $i\in\ccalA_c(t_0)$  (c.f Lemma \ref{lemma_critical_points_new} ). Thus the norm of the gradient of the partial navigation function is bounded below for any $x(t)$ with $t\in[t_0,t_f)$ for a set of initial conditions of measure one. Hence, there exists a constant $L>0$ such that 
\begin{equation}\label{eqn_lower_bound_gradient_aux}
\left\| \nabla \varphi_{k,\ccalA_c(t_0)}(x(t)) \right\| \geq L, \forall t\in[t_0,t_f). 
\end{equation} 
From the fundamental theorem of calculus we can write
\begin{equation}\label{eqn_fund_theo_calc}
\varphi_{k,\ccalA_c(t_0)}(x(t_f)) -\varphi_{k,\ccalA_c(t_0)}(x(t_0)) = \int_{t_0}^{t_f} \dot{\varphi}_{k,\ccalA_c(t_0)}(x(s)) ds.
\end{equation}
Write the right hand side of the above equation as
\begin{equation}
\int_{t_0}^{t_f} \dot{\varphi}_{k,\ccalA_c(s)}(x(s))ds = \int_{t_0}^{t_f} \nabla \varphi_{k,\ccalA_c(t_0)}^T(x(s)) \dot{x} d{s}
\end{equation}
and substitute $\dot{x}$ by the expression in \eqref{eqn_partial_gradient_flow} 
\begin{equation}
\int_{t_0}^{t_f} \dot{\varphi}_{k,\ccalA_c(s)}(x(s)) ds= -\int_{t_0}^{t_f} \left\| \nabla \varphi_{k,\ccalA_c(t_0)}(x(s)) \right\|^2 d{s}.
\end{equation}
Finally combine the above expression with \eqref{eqn_fund_theo_calc} and the bound in \eqref{eqn_lower_bound_gradient_aux} to write
\begin{equation}\label{eqn_aux_theorem4}
\varphi_{k,\ccalA_c(t_f)}(x(t_f)) -\varphi_{k,\ccalA_c(t_0)}(x(t_0)) \leq  \int_{t_0}^{t_f} L^2ds.
\end{equation}
By integrating the right hand side of the above expression we get the following upper bound for $t_f$
\begin{equation}\label{eqn_finite_time}
t_f \leq t_0 +\frac{\varphi_{k,\ccalA_c(t_0)}(x(t_0)) -\varphi_{k,\ccalA_c(t_0)}(x(t_f))}{L^2}. 
\end{equation}
Since the navigation function is always bounded (c.f. Definition \ref{def_navigation_function}) the time until the agent visits a new obstacle if finite, which completes the proof of the theorem.

\bibliographystyle{ieeetr}
\bibliography{bib}

\begin{IEEEbiography}{Santiago Paternain}
received the B.Sc. degree in electrical engineering from Universidad de la Rep\'ublica Oriental del Uruguay, Montevideo, Uruguay in 2012. Since August 2013, he has been working toward the Ph.D. degree in the Department of Electrical and Systems Engineering, University of Pennsylvania. His research interests include optimization and control of dynamical systems. 
\end{IEEEbiography}
  \begin{IEEEbiography}{Daniel E. Koditschek}
  (S’80–M’83–SM’93–F’04) received the Ph.D. degree in electrical engineering from Yale University, New Haven, CT, USA, in
1983. Koditschek is the Alfred Fitler Moore Professor of Electrical and Systems Engineering at the University of Pennsylvania, where he served as Chair from 2005 - 2012. He holds secondary appointments in the Departments of Computer and Information Science and Mechanical Engineering and Applied Mechanics. Prior to joining Penn, Koditschek held faculty positions in the Electrical Engineering and Computer Science Department at the University of Michigan, Ann Arbor (1993-2004) and the Electrical Engineering Department at Yale University (1984 - 1992). Koditschek’s current research interests include robotics, the application of dynamical systems theory to intelligent machines, and nonlinear control. Dr. Koditschek is a member of the AMS, ACM, MAA, SIAM, SICB, and Sigma Xi. He is a fellow of the AAAS. He received the Presidential Young Investigators Award in 1986.
\end{IEEEbiography}
  \begin{IEEEbiography}{Alejandro Ribeiro}
received the B.Sc. degree in electrical engineering from the Universidad de la Republica Oriental del Uruguay, Montevideo, in 1998 and the M.Sc. and Ph.D. degree in electrical engineering from the Department of Electrical and Computer Engineering, the University of Minnesota, Minneapolis in 2005 and 2007. From 1998 to 2003, he was a member of the technical staff at Bellsouth Montevideo. After his M.Sc. and Ph.D studies, in 2008 he joined the University of Pennsylvania (Penn), Philadelphia, where he is currently the Rosenbluth Associate Professor at the Department of Electrical and Systems Engineering. His research interests are in the applications of statistical signal processing to the study of networks and networked phenomena. His focus is on structured representations of networked data structures, graph signal processing, network optimization, robot teams, and networked control. Dr. Ribeiro received the 2014 O. Hugo Schuck best paper award, the 2012 S. Reid Warren, Jr. Award presented by Penn's undergraduate student body for outstanding teaching, the NSF CAREER Award in 2010, and paper awards at the 2016 SSP Workshop, 2016 SAM Workshop, 2015 Asilomar SSC Conference, ACC 2013, ICASSP 2006, and ICASSP 2005. Dr. Ribeiro is a Fulbright scholar and a Penn Fellow.
\end{IEEEbiography}

\end{document}